\documentclass[11pt]{article}
\usepackage{graphicx} 

\title{Singular Backward SDEs for Optimal Control with State Constraints}
\author{Mathieu Lise\thanks{New York University, Tandon School of Engineering. This work is partially supported by NSF grant $\#$DMS-2508581.}~\footnote{Email: mathieu.lise@nyu.edu} \and Nizar Touzi\footnotemark[1]~\footnote{Email: nizar.touzi@nyu.edu}}
\date{}

\usepackage[backend=biber, maxnames=99]{biblatex}

\AtBeginBibliography{\fontsize{10}{13}\selectfont}

\usepackage{geometry}
\usepackage{hyperref}
\hypersetup{
    colorlinks=true,
    linkcolor=blue,
    citecolor=magenta,
    filecolor=magenta,      
    urlcolor=cyan,
    pdftitle={Overleaf Example},
    pdfpagemode=FullScreen,
    }

\newcommand{\eqrange}[2]{(\ref{#1}--\ref{#2})}

\usepackage{setspace}
\usepackage{amsmath}
\usepackage{amsfonts}
\usepackage{mathtools}
\usepackage{mathrsfs}
\usepackage{amssymb}
\usepackage{tabularx}
\usepackage{algorithm}
\usepackage{array}
\usepackage{graphicx}
\usepackage{multirow}
\usepackage{amsthm}
\usepackage{algorithmic}
\usepackage{xcolor}
\usepackage{bm}
\usepackage{float}
\usepackage[overload]{empheq}
\usepackage{hyperref}
\usepackage{tabularx}
\usepackage{bbm}
\usepackage{enumitem}

\newtheorem{theorem}{Theorem}
\newtheorem{lemma}[theorem]{Lemma}
\newtheorem{proposition}[theorem]{Proposition}
\newtheorem{definition}[theorem]{Definition}

\theoremstyle{definition}
\newtheorem{rmk}[theorem]{Remark}
\newtheorem{example}[theorem]{Example}
\newtheorem{assumption}{Assumption}
\newtheorem{corollary}[theorem]{Corollary}

\numberwithin{equation}{section}
\numberwithin{theorem}{section}

\newcounter{assumptiontag}
\renewcommand{\theassumptiontag}{A\arabic{assumption}.\roman{assumptiontag}}

\begin{document}
\maketitle

\setlength{\parindent}{0pt}
\begin{spacing}{1.15}

\begin{abstract}
    We investigate a class of backward stochastic differential equations (BSDEs) with at most quadratic growth which explode at a possibly unbounded random horizon, defined through the first hitting of zero of an adapted Itô process. In contrast with the classical theory of singular BSDEs, the explosion is generated by the nonlinear dependence of the generator on the martingale integrand, rather than by a superlinear coercivity condition in the solution component. We construct a minimal singular solution and derive two-sided estimates on its explosion, together with weighted BMO estimates for the martingale integrand. We also obtain the uniqueness and exact explosion rates under additional structural assumptions. For Hamiltonian generators, we establish a verification theorem for an infinite-horizon stochastic optimal control problem with possible non-Markovian state constraints, showing that the BSDE feedback induces the unique optimal constrained law. We finally specialize the theory to exit times of uniformly elliptic Markov diffusions and recover the connection with large solutions of viscous Hamilton-Jacobi equations.
\end{abstract}


\vspace{15pt}


\vspace{5pt}

\section{Introduction}

\setlength{\parindent}{20pt}

State constraints in stochastic control require the controlled state to remain in a prescribed admissible region for all times. For non-degenerate diffusions, such constraints are intrinsically singular: in general they cannot be enforced by bounded feedback controls. In a Markovian framework, this singular behavior appears naturally in the associated Hamilton–Jacobi–Bellman equation through an infinite boundary condition. A canonical example was studied by \citeauthor{Lasry_Lions_state_constraints} \cite{Lasry_Lions_state_constraints}. For $1 < p \leq 2 $, they considered viscous Hamilton-Jacobi equations of the form:
\begin{equation}\label{HJB_singular}
\begin{array}{cc}
    \displaystyle  -\frac{1}{2} \Delta u + \frac{1}{p} \big|\nabla u \big|^p + \lambda u = h(x)  & \text{in } \;  D \, ,
\end{array}
\end{equation}
in a domain $D$, together with the boundary blow-up condition:
\begin{equation*}\label{eq:singular_boundary_condition}
    \begin{array}{cc}
            \displaystyle u(x) \longrightarrow +\infty \,, & \text{as } \, x \to \partial D   \, .
        \end{array}
\end{equation*}
For a suitable class of data, and in particular for bounded $h$, they established existence, uniqueness, precise boundary estimates, and a stochastic control interpretation of the so-called large solution. The singular behavior of the solution near the boundary produces an unbounded optimal feedback which prevents the optimally controlled diffusion from leaving the domain.

The purpose of this paper is to develop a probabilistic and possibly non-Markovian counterpart of this construction. Rather than describing the constraint only through the position of a Markov diffusion relative to a fixed domain, we introduce an adapted Itô process and a stopping time:
\begin{equation*}
    \begin{array}{ccc}
         \displaystyle  \delta_t = \delta_0 + \int_0^t \Xi_s ds + \int_0^t U_s \cdot dW_s  & \text{and } & \displaystyle  {\rm{T}}^{\delta} \coloneqq \inf \big\{ t \geq 0, \; \delta_t \leq 0 \big\} \, .
    \end{array}
\end{equation*}
The process $\delta$ plays the role of a stochastic distance to the constraint. Its coefficients may depend progressively on the past, so the resulting constraint need not be Markovian. Classical exit times from domains are included as a special case.

The natural equation associated with this problem is a backward stochastic differential equation (BSDE for short) posed on $[0,{\rm{T}}^{\delta})$ and satisfying the singular terminal condition:
\begin{equation}\label{BSDE_LL}
            \begin{array}{cccc}
                 \displaystyle d Y_t =  -f_t(Y_t, Z_t) \, dt + Z_t \cdot d W_t \, , &  t < {\rm{T}}^{\delta} \, , & \text{and} &
                 \displaystyle \lim_{t \rightarrow {\rm{T}}^{\delta}} Y_t = +\infty  \; \; \text{ on } \; \{ {\rm{T}}^{\delta} < \infty \} \, ,
            \end{array}
\end{equation}
We consider generators which are monotone in $y$ and have $p$-polynomial growth in $z$, with $1 < p \leq 2$. A representative example is:
\begin{equation}\label{eq:ex_generator}
    f_t(y, z) = -\frac{1}{p}|z|^p + h_t - \lambda y \, ,
\end{equation}
for some progressively measurable process $(h_t)_{t \ge 0}$. In this framework, the explosion results from the nonlinear dependence on $Z$, combined with the approach of the stochastic distance $\delta$ to zero.

This mechanism is different from the one underlying most of the existing literature on singular backward SDEs. In the work initiated by \citeauthor{popier} \cite{popier} and developed in many subsequent papers, see \textit{e.g.} \cite{popier_random_time, popier_random_time3, popier_general_driver}, singular terminal conditions are typically treated under a superlinear monotonicity assumption in $Y$, closely related to the Keller-Osserman condition for elliptic equations with boundary blow-up. Such BSDEs have also been used in stochastic control problems with terminal state constraints, notably in optimal liquidation, see \citeauthor{kruse_control} \cite{kruse_control} and \citeauthor{popier_non_markovian} \cite{popier_non_markovian}. Our setting is complementary: here the singularity is generated by the gradient, or Hamiltonian, part of the equation. The relevant barriers are therefore closer in spirit to those introduced by Lasry and Lions for Equation \eqref{HJB_singular}. At the same time, the possibly quadratic growth in $Z$ requires tools from the theory of quadratic backward SDEs, initiated by \citeauthor{kobylanski} \cite{kobylanski} and \citeauthor{briand_hu} \cite{briand_hu, briand_hu2}, in particular bounded mean oscillations (BMO) estimates and Girsanov transformations.

Our first main result establishes the existence of a minimal singular solution and provides quantitative estimates on its explosion. For $1 < p <2$, the natural blow-up scale is $d^{-r}$, where $r = \frac{2-p}{p-1}$, while if $p=2$ it is in $-\log(d)$. We also derive weighted BMO estimates for the martingale integrand $Z$, which may be viewed as stochastic counterparts of the boundary gradient estimates satisfied by large solutions of \eqref{HJB_singular}. An important ingredient in our analysis is a comparison principle for backward SDEs with quadratic growth on an arbitrary random horizon, in line with the now standard results of \citeauthor{pardoux_darling} \cite{pardoux_darling}, \citeauthor{royer} \cite{royer}, \citeauthor{nizar_2ndBSDE} \cite{nizar_2ndBSDE}, and \citeauthor{QBSDE_random_time} \cite{QBSDE_random_time}. Our proof does not require any specific structure of the terminal time, and carries the singular terminal value $+\infty$. Under additional assumptions, we also identify the exact asymptotic behavior of the process $Y$ when $t$ approaches the terminal time.

When the nonlinearity $f$ is concave in $z$, we obtain the uniqueness of solutions of \eqref{BSDE_LL} and the connection with a stochastic optimal control problem with state constraints. In the context of Example \eqref{eq:ex_generator}, we introduce the discounted criterion:
\begin{equation}\label{lasry_lions_cost_function}
    J(x, \alpha) = \mathbb{E}^{\mathbb{P}^{\alpha}} \hspace{-3pt} \int_0^{\infty} e^{-\lambda t} \Big( \frac{1}{p'}|\alpha_t|^{p'} + h_t \Big) \,  dt \, ,
\end{equation}
where $p' = \frac{p}{p-1}$ is the conjugate of $p$ and $\mathbb{P}^{\alpha}$ is the probability measure on the path space under which the canonical process $X$ has the decomposition:
\begin{equation*}
\begin{array}{cccc}
     dX_t = -\alpha_t \,  dt + dW_t^{\alpha} & \text{and} & X_0 = x & \text{for some } \, \mathbb{P}^{\alpha}-\text{Brownian motion } \, W^{\alpha} \, .
\end{array}
\end{equation*}
Here, the admissible strategies are progressively measurable processes $\alpha$ with suitable integrability, which force the process $\delta$ never to hit 0, \textit{i.e.} $\mathbb{P}^{\alpha}[{\rm{T}}^{\delta} = \infty] = 1$. We identify an optimal feedback control $\hat{\alpha}$ and prove a verification argument. It turns out that the associated constrained probability measure $\mathbb{P}^{\hat{\alpha}}$ is the unique optimal admissible law.

This control formulation is related to finite-horizon constrained problems which, in the quadratic case, admit an entropic interpretation through changes of measure and conditioned diffusions, see \citeauthor{fuhrman_states_constraints} \cite{fuhrman_states_constraints}. Similar variational structures appear for example in Schrödinger bridges. The present framework differs in several ways: the horizon is infinite and discounted, the Hamiltonian may have sub-quadratic growth, and both the stochastic distance defining the constraint and the running cost may be path-dependent. In particular, outside the quadratic case there is no entropic linearization.

We then specialize the abstract theory to exit times of uniformly elliptic Markov diffusions, taking $\delta$ to be a regularized distance to the boundary of a smooth domain $D$. In particular, when the state process $X$ is a Brownian motion and the generator $f$ is Markovian, the singular backward SDE \eqref{BSDE_LL} can be identified with the large solution of the associated viscous Hamilton-Jacobi equation through:
\begin{equation*}
    \begin{array}{ccc}
         Y_t = u(X_t) & \text{and} & Z_t = \nabla u(X_t)  \, ,
    \end{array}
\end{equation*}
showing the consistency of our results with the framework of \citeauthor{Lasry_Lions_state_constraints} \cite{Lasry_Lions_state_constraints}. 

Our motivation goes beyond the theoretical study of singular backward SDEs. We are also interested in developing non-Markovian counterparts of classical diffusion models with constrained state space. A representative example is the family of positive square-root diffusions, including squared-Bessel and Cox–Ingersoll–Ross type models. In the Markovian framework with $D = \mathbb{R}_+^*$, the Bessel dynamics appear formally in the ergodic limit $\lambda \to 0^+$. Although our present estimates do not allow us to justify this limit in the general setting, the discounted state-constrained models constructed here can be viewed as non-Markovian analogues of these classical positive diffusions. Possible applications include structure preserving generative diffusion models, in the spirit of \citeauthor{tweedie_formula_nizar} \cite{tweedie_formula_nizar}. These models aim at generating samples from a probability distribution through a diffusion process which remains in the state space dictated by the structure of the data. 

The paper is organized as follows. In Section \ref{notations_main_results_sec}, we introduce our framework and notations before stating the main results in Section \ref{sec:main_results}. In Section \ref{existence_sec}, we prove our main existence and uniqueness theorems for the singular backward SDEs \eqref{BSDE_LL}. We establish the connection with stochastic optimal control with state constraints in Section \ref{Control_sec}. Finally, in Section \ref{Applications_sec} we study the application to a Markovian state constrained problem.

\section{General framework} \label{notations_main_results_sec}

In this section, we introduce some notations and the assumptions we will use in the rest of the paper.

\subsection{Notations}

\setlength{\parindent}{10pt}

We first recall some classical definitions from functional analysis. Let $D \subseteq \mathbb{R}^d$ be an open connected domain and denote by $\overline{D} = \text{cl}(D)$ its closure. Let $L^0(D)$, $L^1(D)$ be respectively the sets of Lebesgue-measurable and integrable functions $u : D \to \mathbb{R}$. We define:

\begin{itemize}
    \item $W^{2, k}(D)$ the set of measurable functions $u \in L^0(D)$ such that $|\partial^{\alpha}u|^k \in L^1(D) $ for all $\alpha \in \mathbb{Z}_+^d$ with $|\alpha| \leq 2$. We say that $u \in W^{2, k}_{\text{loc}}(D)$ if $u \in W^{2, k}(K)$ for any compact $K \subset D$.

    \item $\mathscr{C}_d \coloneqq C([0, \infty), \mathbb{R}^d)$ the space of continuous functions from $[0, \infty)$ to $\mathbb{R}^d$. For any ${\rm{x}} \in \mathscr{C}_d$ and $t \geq 0$, we note ${\rm{x}}_{\wedge t } = ({\rm{x}}_{s \wedge t })_{s \geq 0}$.

    \item $d_{\pm}$ the signed distance to the boundary $\partial D$ and $D^{\eta}$ the tubular neighborhood of $\partial D$ of size $2\eta$:
        \begin{equation*}\label{signed_dist}
        \begin{array}{ccc}
             \displaystyle d_{\pm}(x) \coloneqq \begin{cases}
                \begin{array}{cc}
                {\rm{dist}}(x, \partial D) & \text{if } \; x \in D \\
                - {\rm{dist}}(x, \partial D) & \text{if } \; x \notin D \, ,
                \end{array} 
            \end{cases} & \text{and} & \displaystyle D^{\eta} \coloneqq \big\{ x \in \mathbb{R}^d , \; |d_{\pm}(x)| < \eta \big\} \, ,
        \end{array}
        \end{equation*}
        where ${\rm{dist}}$ denotes the Euclidean distance in $\mathbb{R}^d$.
\end{itemize}

If there exists $\overline{\eta} > 0$ such that $d_{\pm} \in C^2(D^{\overline{\eta}})$, which is the case when $D$ is a bounded domain with a $C^2$-regular boundary, see \cite[Lemma~14.16]{Boundary_distance}, we often consider a regularized and truncated distance function:

\vspace{-10pt}

\begin{equation}\label{regularized_distance}
\begin{array}{cc}
     \mathbf{d} \coloneqq (1 - \phi) d_{\pm} + R \phi \, , & \text{for some } R > 0 \, ,
\end{array}     
\end{equation}
where $\phi \in C^{\infty}(\mathbb{R}^d, [0,1])$ is a smooth function satisfying $\phi \equiv 1$ on $\mathbb{R}^d \backslash D^{\overline{\eta}}$ and $\phi \equiv 0$ on $D^{\overline{\eta} / 2}$. 

Let $(\Omega, \mathcal{F}, \mathbb{P})$ be a probability space with filtration $\mathbb{F} = \{ \mathcal{F}_t, t \ge 0 \}$. We denote by $\mathbb{E}_t \coloneqq \mathbb{E}[\cdot | \mathcal{F}_t]$ the expectation conditional to $\mathcal{F}_t$. Given a normed vector space $A$ and a $\mathbb{F}$-stopping time $\theta$, we define:

\begin{itemize}
    \item $\mathbb{H}^0$ the set of $\mathbb{F}$-progressively measurable processes $\psi : \Omega \times [0, \infty) \rightarrow A$. In the following, we use the standard notation $\psi_t = \psi(\omega, t)$ for any $\psi \in \mathbb{H}^0$.

    \item $\mathbb{H}^{k}_{\theta}$, where $1 \leq k < \infty$, the set of $\psi \in \mathbb{H}^0$ such that $\displaystyle \mathbb{E} \Big( \int_0^\theta |\psi_t|^2 dt \Big)^{k / 2} < \infty$. We write $\psi \in \mathbb{H}^{k, {\rm{loc}}}_{\theta}$ \vspace{-17pt} 
    
    if $\psi \in \mathbb{H}^{k}_{T \wedge \theta}$ for all $T > 0$.

    \item $\mathbb{H}^{\text{BMO}}_{\theta}$ the set of processes $\psi \in \mathbb{H}^0$ such that $\| \psi \|_{{\rm{BMO}}, \theta} \coloneqq \displaystyle \sup_{\kappa \in \mathcal{T}_0^{\theta}} \Big\Vert \mathbb{E}_{\kappa}  \Big(\int_{\kappa}^{\theta} |\psi_t|^2 dt \Big)^{1/2} \, \Big\Vert_{\infty} < \infty$, \vspace{-10pt} 
    
    where $\mathcal{T}_0^{\theta}$ is the set of stopping times $\kappa \leq \theta$, $\mathbb{P}$-a.s. 

    \item $\mathbb{L}^{\infty}(\mathcal{F}_{\theta})$ the set of bounded and $\mathcal{F}_{\theta}$-measurable random variables, and $\mathbb{S}^{\infty}_{\theta}$ the set of pathwise continuous processes $\psi \in \mathbb{H}^0$ such that $\displaystyle  \| \psi_{\wedge \theta} \|_{\infty} < \infty$. 
    
    \item For a local martingale $M$, we denote by $\mathcal{E}(M)_{\theta} \coloneqq \exp (M_{\theta} -\frac{1}{2}\langle M \rangle_{\theta} )$ the stochastic exponential. 

\end{itemize} 

Finally, let $x^+$ and $x^-$ denote respectively $\max(x, 0)$ and $-\min(x, 0)$ if $x\in \mathbb{R}$, and $\ell$ be a smooth truncation at infinity of $-\log$ defined by $\ell$ is smooth, non-increasing and $\ell({\rm{d}}) = -\textnormal{log}({\rm{d}})$ when ${\rm{d}} \in (0, 1/2)$ and $\ell({\rm{d}}) = 0$ when ${\rm{d}} \geq 1$. Such function always exists and there is a constant $L_{\log} > 0$ such that:
\begin{equation}\label{smooth_log_truncation}
\begin{array}{cc}
     \displaystyle |\ell'({\rm{d}}) + 1/{\rm{d}}|+ |\ell''({\rm{d}}) - 1/{\rm{d}}^2| \leq L_{\log} \, , & \text{for all } \, {\rm{d}} > 0 \, .
\end{array}
\end{equation}

Recall the main result of \citeauthor{Lasry_Lions_state_constraints} \cite[Theorem~II.1]{Lasry_Lions_state_constraints} on the large solutions of the nonlinear elliptic partial differential equation \eqref{HJB_singular}. Let $p \in (1, 2]$ and denote by $\Phi_p : (0, \infty) \rightarrow \mathbb{R}$ the function:
\begin{equation}\label{Phi_p}
\begin{array}{cc}
     \displaystyle \Phi_p({\rm{d}}) = {\rm{d}}^{-r} \, \mathbbm{1}_{\{1 < p < 2\}} \, + \, \ell({\rm{d}}) \, \mathbbm{1}_{\{p=2\}} \, , & \text{where } \,\displaystyle  r = \frac{2-p}{p-1} \, .
\end{array}
\end{equation}

\begin{theorem}[\cite{Lasry_Lions_state_constraints}]\label{Lasry_Lions_summary}
    Let $D$ be an open bounded domain of $\mathbb{R}^d$ with $C^2$-regular boundary, $p \in (1, 2]$ and $g \in L^{\infty}(D)$. Then there exists a unique solution $u \in \bigcap_{k \ge 1} W^{2,k}_{\textnormal{loc}}(D)$ of \eqref{HJB_singular} satisfying $u(x)\rightarrow +\infty$ when $x \rightarrow \partial D$. Moreover, $\| \mathbf{d}^{r+1}\nabla u \|_{L^{\infty}(D)} < \infty$ and $u$ satisfies near the boundary:
    \begin{equation*}\label{limit_u_boundary_subquad}
        \frac{u(x)}{\Phi_p \! \circ  \mathbf{d}(x)} \xrightarrow[d(x) \rightarrow 0]{} C_0  \coloneqq
        \frac{(p-1)^{-\frac{2-p}{p-1}}}{2-p} \big( \frac{p}{2} \big)^{\frac{1}{p-1}}  \mathbbm{1}_{\{1 < p < 2\}} + \mathbbm{1}_{\{p=2\}} \, .
    \end{equation*}
\end{theorem}

This result was later proved to hold in the specific unbounded domain $D = \mathbb{R}_+^* \times \mathbb{R}^{d-1}$ in the appendix of \citeauthor{porretta_veron} \cite{porretta_veron}. 

\subsection{The path-dependent setting}

Let $W = (W_t)_{t \geq 0}$ be a $\mathbb{R}^d$-valued Brownian motion on $(\Omega, \mathcal{F}, \mathbb{P})$. Denote by $\mathbb{F}  = (\mathcal{F}_t)_{t \geq 0}$ the $\mathbb{P}$-completion of the filtration generated by $W$. Consider the $\mathbb{R}$-valued Itô process:
\begin{equation}\label{delta_SDE}
    \begin{array}{cc}
         \displaystyle  \delta_t = \delta_0 + \int_0^t \Xi_s ds + \int_0^t U_s \cdot dW_s \, , & \text{for some } \, \delta_0 > 0 \, ,
    \end{array}
\end{equation}
where $\Xi$ and $U$ are measurable and $\mathbb{F}$-adapted processes. We denote by ${\rm{T}}_{\eta}^{\delta}$ the $\mathbb{F}$-adapted stopping time: 
\begin{equation}\label{tau_eta}
    \begin{array}{cccc}
         {\rm{T}}^{\delta}_{\eta} \coloneqq \inf \big\{ t \geq 0, \; \delta_t \leq \eta \big\} & \text{for } \, \eta \in \mathbb{R} \, , & \text{and} & {\rm{T}}^{\delta} \coloneqq {\rm{T}}^{\delta}_0 \, .
    \end{array}
\end{equation}

\vspace{1pt}

\begin{assumption}\label{d_assumptions} The processes $\Xi : \Omega \times [0, \infty) \rightarrow \mathbb{R}$ and $U : \Omega \times [0, \infty) \rightarrow \mathbb{R}^{d}$ satisfy, for some $\overline{\eta} > 0$,
    \begin{itemize}
        \refstepcounter{assumptiontag}
        \item[] \hspace{-30pt} (\theassumptiontag) \label{Xi_bounded} \hspace{1pt} There exist constants $m, L_{\Xi} > 0$ such that:
        \begin{equation*}
            \begin{array}{cc}
                 \displaystyle |\Xi_t| \leq L_{\Xi} (1 + |\delta_t|^m) \, , & \mathbb{P}-\text{a.s. on } \, \{t \leq {\rm{T}}^{\delta}_{-\overline{\eta}} \} \, .
            \end{array}
        \end{equation*}

        \refstepcounter{assumptiontag}
        \item[] \hspace{-30pt} (\theassumptiontag) \label{delta_coef_boundedness} \hspace{1pt} There exist $0 < \sigma_0 \leq \sigma_1$ such that:
        \begin{equation*}
            \begin{array}{cccc}
               |U_t| \leq \sigma_1 \,\text{ on } \, \{t \leq {\rm{T}}^{\delta}_{-\overline{\eta}} \} \,\text{ and } \, |U_t| \geq \sigma_0 \, \text{ on } \, \{ -\overline{\eta} \leq \delta_t \leq \overline{\eta} \, \} \,,  &\mathbb{P}-\text{a.s.}
            \end{array}
        \end{equation*}
        
    \end{itemize}
\end{assumption}

\begin{rmk}
    By introducing $\delta$, we first aim at generalizing $\mathbf{d}(W_{\cdot})$ in the setting of \citeauthor{Lasry_Lions_state_constraints}, which is the regularized distance to the boundary of the state process in Theorem \ref{Lasry_Lions_summary}. Recall that if $\partial D$ is smooth then $\nabla \mathbf{d}$ and $\Delta \mathbf{d}$ are bounded near the boundary. In particular, our formulation allows to replace $\mathbf{d}$ by a so-called anisotropic distance, see \citeauthor{finsler_state_constraints} \cite{finsler_state_constraints}. Moreover, one could imagine a domain $D(\omega, t)$ whose shape varies over time and with the Brownian noise.
\end{rmk}

\begin{rmk}
    Assumption \ref{d_assumptions} does not require that $\mathbb{P}[{\rm{T}}^{\delta} < \infty] = 1$. In particular this is not true for $d=1$ and $\delta_t = \delta_0 + t + W_t$. Indeed, if $\mathbb{P}[{\rm{T}}^{\delta} < \infty] = 1$, then by the Markov property of $W$, there would be a sequence of stopping times $\theta_n \rightarrow \infty$ such that $\delta_{\theta_n} = 0$ a.s., contradicting the fact that $\delta_t \rightarrow \infty$ a.s. \hspace{-7pt} as $t \rightarrow \infty$.
\end{rmk}

Our assumptions on the driver $f : \Omega \times [0, \infty) \times \mathbb{R} \times \mathbb{R}^d \rightarrow \mathbb{R}$ are inspired from \citeauthor{Pozza_thesis} \cite{Pozza_thesis} and slightly generalize those of \citeauthor{Lasry_Lions_state_constraints} \cite{Lasry_Lions_state_constraints}.

\begin{assumption}\label{driver_assumptions}
    \setcounter{assumptiontag}{0}
    
    \begin{itemize}
        \item[]

        \refstepcounter{assumptiontag}
        \item[] \hspace{-30pt} (\theassumptiontag) \label{f_prog_meas} \hspace{1pt} The function $f$ is $\mathbb{F}$-progressively measurable and $(y,z) \mapsto f_t(y,z)$ is continuous for all $t \geq 0$. Moreover, there exists a constant $C_f > 0$ such that, for all $(y, z) \in \mathbb{R} \times \mathbb{R}^{d}$,
        \begin{equation*}
        \begin{array}{cc}
             |f_t(y,z) - f_t(y,z')| \leq C_f \big( 1 + |z| + |z'| \big) |z-z'| \, , &  d\mathbb{P} \otimes dt-\text{a.e.}
        \end{array}
        \end{equation*}

        \refstepcounter{assumptiontag}
        \item[] \hspace{-30pt} (\theassumptiontag)\label{f_monotonicity} \hspace{1pt} There exists $\lambda_0 > 0$ such that, for all $(y, z) \in \mathbb{R} \times \mathbb{R}^{d}$,
        \begin{equation*}
            \begin{array}{cc}
            (y-y') (f_t(y,z) - f_t(y',z)) \leq - \lambda_0 |y-y'|^2 \, , &  d\mathbb{P} \otimes dt-\text{a.e.}
        \end{array}
        \end{equation*}

        \refstepcounter{assumptiontag}
        \item[] \hspace{-30pt} (\theassumptiontag) \label{p_growth} \hspace{1pt} There exists $p \in (1, 2]$, $k_0 \geq k_1 > 0$, $\lambda_1 \geq \lambda_0$ and $\vartheta > 0$ such that the constant $m$ introduced in Assumption \ref{d_assumptions} satisfies $m \leq \frac{1}{p-1}$ and
        \begin{equation*}
            \begin{array}{cc}
                 \displaystyle - k_0 |z|^p - \lambda_1 |y| - \vartheta \leq f_t(y, z) \leq - k_1 |z|^p + \lambda_1 |y| + \vartheta \, , & \text{for all }\, (y, z) \in \mathbb{R} \times \mathbb{R}^{d}\, , \,  d\mathbb{P} \otimes dt-\text{a.e.}
            \end{array}
        \end{equation*}
        
    \end{itemize}
\end{assumption}

Assumptions \eqrange{f_prog_meas}{f_monotonicity} are natural in the literature on quadratic BSDEs and essential for applying known results on random horizon quadratic backward SDEs, see \citeauthor{QBSDE_random_time} \cite{QBSDE_random_time}. Assumption \eqref{p_growth} is needed to adapt the a priori estimates of \citeauthor{Lasry_Lions_state_constraints} \cite{Lasry_Lions_state_constraints} and we will use it to prove the existence of a solution to the singular BSDE \eqref{BSDE_LL}.

\section{Main results}\label{sec:main_results}

We state here our main results, and we postpone most of the proofs to the remaining of the paper.

\subsection{Singular backward SDE with controlled explosion}

We first define the notion of solutions for the backward SDE \eqref{BSDE_LL} with singular terminal condition, which we call ${\rm{BSDE}}_{\infty}({\rm{T}}^{\delta}, f)$. The main idea is to localize through the level hitting times ${\rm{T}}^{\delta}_{\eta}$. Solutions remain bounded on every pre-singular level set, and supersolutions are allowed to explode at a limiting level. This formulation is stable under comparison and is convenient for the construction. Note that ${\rm{T}}^{\delta}_{\eta} \rightarrow {\rm{T}}^{\delta}_{\overline{\eta}}$ when $\eta \rightarrow \overline{\eta}$.

\begin{definition}\label{singular_def}
    Let $\theta$ be a $\mathbb{F}$-stopping time and $(Y, Z)$ an adapted process with values in $\mathbb{R} \times \mathbb{R}^d$. Up to the integrability of $Y$ and $Z$ which we explicit below, we define the continuous process:

    \vspace{-5pt}
    
        \begin{equation*}
            \begin{array}{cc}
                 \displaystyle K_t \coloneqq Y_0 - Y_t - \int_0^t f_s(Y_s, Z_s) \, ds  + \int_0^t Z_s \cdot dW_s \, , & t < \theta \, .
            \end{array}
        \end{equation*}
        
    \begin{enumerate}[leftmargin=.5cm]
        \item If $\xi \in \mathbb{L}^{\infty}(\mathcal{F}_{\theta})$ and $(Y, Z) \in \mathbb{S}^{\infty}_{\theta}(\mathbb{R}) \times \mathbb{H}^{2, {\rm{loc}}}_{\theta}(\mathbb{R}^d)$, we say that:
        \begin{itemize}[leftmargin=.3cm]
           \item[a.] $(Y, Z)$ is a supersolution of\hspace{.5pt} ${\rm{BSDE}}_{\xi}(\theta, f)$ if $K$ is non-decreasing and $\displaystyle \liminf_{t \to \theta } Y_{t} \geq \xi$ on $\{ \theta < \infty\}$.

           \vspace{5pt}

            \item[b.] $(Y, Z)$ is a subsolution of\hspace{.5pt} ${\rm{BSDE}}_{\xi}(\theta, f)$ if $K$ is non-increasing and $\displaystyle \limsup_{t \to \theta } Y_{t} \leq \xi$ on $\{ \theta < \infty\}$.
        \end{itemize}

        \item If $\theta = {\rm{T}}^{\delta}$ and $\xi \equiv +\infty$, we say that:
        \begin{itemize}[leftmargin=.3cm]
            \item[a.] $(Y, Z)$ is a supersolution of\hspace{.5pt} ${\rm{BSDE}}_{\infty}({\rm{T}}^{\delta}, f)$ if there exists $\eta_{\textnormal{sup}} \geq 0$ such that, for all $\eta > \eta_{\textnormal{sup}}$, $\| Y_{{\rm{T}}^{\delta}_{\eta}} \|_{\infty} < \infty$, $(Y, Z)$ is a supersolution of\hspace{.5pt} ${\rm{BSDE}}_{Y_{{\rm{T}}^{\delta}_{\eta}}}({\rm{T}}^{\delta}_{\eta}, f)$ and $\displaystyle \lim_{\eta \rightarrow \eta_{\textnormal{sup}} } Y_{{\rm{T}}^{\delta}_{\eta}} = +\infty$ on $\{ {\rm{T}}^{\delta}_{\eta_{\textnormal{sup}}} \! \!< \! \infty\}$.

            \vspace{5pt}

            \item[b.] $(Y, Z)$ is a subsolution of\hspace{.5pt} ${\rm{BSDE}}_{\infty}({\rm{T}}^{\delta}, f)$ if there exists $\eta_{\textnormal{sub}} \leq 0$ such that, for all $\eta > \eta_{\textnormal{sub}}$, $\| Y_{{\rm{T}}^{\delta}_{\eta}} \|_{\infty} < \infty$ and $(Y, Z)$ is a subsolution of\hspace{.5pt} ${\rm{BSDE}}_{Y_{{\rm{T}}^{\delta}_{\eta}}}({\rm{T}}^{\delta}_{\eta}, f)$.
        \end{itemize}
        
        \item We say that $(Y,Z)$ is a solution of\hspace{.5pt} ${\rm{BSDE}}_{\xi}(\theta, f)$ or\hspace{.5pt} ${\rm{BSDE}}_{\infty}({\rm{T}}^{\delta}, f)$ if it is both a subsolution and a supersolution. We say that $(Y, Z)$ is the minimal solution if $Y \leq Y'$, $d\mathbb{P} \otimes dt-$a.e. \hspace{-7pt} for any other solution $(Y', Z')$.
    \end{enumerate}
\end{definition}

\begin{rmk}\label{stopping_time_def_rmk}
    If $(Y, Z)$ is a supersolution of ${\rm{BSDE}}_{\infty}({\rm{T}}^{\delta}, f)$ and $\tau \coloneqq \inf \{ t \geq 0 , \, Y_t = +\infty \}$, then $\tau = {\rm{T}}^{\delta}_{\eta_{\textnormal{sup}}}$ a.s. First, it is clear that $\tau \leq {\rm{T}}^{\delta}_{\eta_{\textnormal{sup}}}$. Now let $t \geq 0$ and notice that:
    \begin{equation*}
    \begin{array}{cc}
         \displaystyle \{ \forall s \leq t, \,  \delta_t > \eta_{\textnormal{sup}} \} = \bigcup_{N \in \mathbb{N}} \big\{ \forall s \leq t, \, \delta_t \geq \eta_{\textnormal{sup}} + 2^{-N} \big\} \subseteq \{ \forall s \leq t, \, Y_s < +\infty\}  \, , & \mathbb{P}-\text{a.s.}
    \end{array}
    \end{equation*}

    \vspace{-8pt}
    
    \noindent This implies that $\mathbb{P}[{\rm{T}}^{\delta}_{\eta_{\textnormal{sup}}} > t ] \leq \mathbb{P}[ \tau > t ]$ and we deduce that $\tau \geq {\rm{T}}^{\delta}_{\eta_{\textnormal{sup}}}$ a.s.
\end{rmk}

As mentioned in the introduction, the construction of solutions to \eqref{BSDE_LL} is based on the estimates derived by \citeauthor{Lasry_Lions_state_constraints} in \cite{Lasry_Lions_state_constraints} and requires a comparison principle for backward SDEs with random horizon and quadratic growth. A similar comparison result was introduced by \citeauthor{kobylanski} \cite[Theorem~2.6]{kobylanski} when the terminal time $\theta$ is bounded or $\mathbb{P}-$a.s. finite. Our proof of Lemma \ref{comp_principle} below involves the special integrability of the process $Z$ (namely bounded mean oscillations) and holds for any terminal time $\theta$. It is inspired from the result of \citeauthor{QBSDE_random_time} \cite[Lemma~3.4]{QBSDE_random_time}.

\begin{lemma}[Comparison Principle]\label{comp_principle}
    Let $\theta$ be a $\mathbb{F}-$stopping time, $\xi, \xi'$ be $\mathcal{F}_{\theta}$-measurable random variables with $\| \xi \|_{\infty} < \infty$ and satisfying either one the following conditions:
    \begin{equation*}
        \begin{array}{ccc}
             \rm{(\ast)} \; \;  \| \xi' \|_{\infty} < \infty \, , & \text{or} & \rm{(\ast \ast)} \; \; \theta = {\rm{T}}^{\delta} \, \text{ and } \, \xi' = +\infty \, .
        \end{array}
    \end{equation*}
    Let $f, f'$ be functions satisfying Assumptions {\rm\eqrange{f_prog_meas}{f_monotonicity}}. Let $(Y,Z)$ be a subsolution of\hspace{.5pt} ${\rm{BSDE}}_{\xi}(\theta, f)$ and $(Y',Z')$ be a supersolution of\hspace{.5pt} ${\rm{BSDE}}_{\xi'}(\theta, f')$. Suppose that $(f_t - f'_t)(Y_t, Z_t) \leq 0$ and $\xi \leq \xi'$, $d\mathbb{P} \otimes dt-$a.e. Then $Y \leq Y'$ a.s. Moreover, if $Y_0 = Y_0'$ then $Y = Y'$ a.s.
\end{lemma}

The proof of Lemma \ref{comp_principle}, reported in Section \ref{BMO_comp_principle_sec}, essentially relies on a Girsanov transformation and does not require any condition on the terminal time $\theta$ nor the specific form of the generator we consider here, in particular Assumption \eqref{p_growth} is not needed.

\begin{rmk}\label{bounded_below_rmk}
      Let $(Y, Z)$ be a solution of \eqref{BSDE_LL}, then $ Y \geq \frac{\overline{m}}{\lambda_0}$ where $ \overline{m} \coloneqq  ({\rm{ess}}\inf f_t(0, 0))^-$ a.s. Indeed, defining $(y_t, z_t) \coloneqq \big(\frac{\overline{m}}{\lambda_0}, 0 \big)$ and using the monotonicity of $f$ in $y$ of Assumption \eqref{f_monotonicity} together with $y_t = \frac{\overline{m}}{\lambda_0} \leq 0$, one obtains $f_t(y_t, z_t) \geq f_t(0, 0) - \lambda_0 y_t \geq 0$. Therefore $(y, z)$ is a bounded subsolution of the singular BSDE \eqref{BSDE_LL}. The claim follows from the comparison result of Lemma \ref{comp_principle}.
\end{rmk}

\begin{rmk}
    The growth condition of Assumption \eqref{p_growth} is essential for the existence of a solution of\hspace{.5pt} ${\rm{BSDE}}_{\infty}({\rm{T}}^{\delta}, f)$ in the sense of Definition \ref{singular_def}. For example, take $f_t(y, z) = -|z| -\lambda y$ for some $\lambda > 0$ and let $(Y, Z)$ be a solution of \eqref{BSDE_LL}. Suppose also that $\mathbb{E}[e^{{\rm{T}}^{\delta} / 2}] < \infty$. Then, applying Itô's formula up to the bounded stopping time $\theta_{\eta} \coloneqq {\rm{T}}^{\delta}_{\eta} \wedge \eta^{-1}$ for some $\eta > 0$,

    \vspace{-2pt}
    
    \begin{equation*}
        \begin{array}{ccc}
            \displaystyle Y_0 = e^{-\lambda \theta_{\eta}} Y_{\theta_{\eta}} - \int_0^{\theta_{\eta}} e^{-\lambda s} Z_s \cdot \big( dW_s + \beta_s ds \big) \, , & \text{where} & \displaystyle \beta_s = \frac{Z_s}{|Z_s|} \mathbbm{1}_{Z_s \neq 0} \, .
        \end{array}
    \end{equation*}

    \vspace{-2pt}
    
    \noindent Notice that the Novikov criterion $\mathbb{E} \big[ e^{\frac{1}{2} \int_0^{{\rm{T}}^{\delta}} |\beta_s|^2 ds} \big] \leq \mathbb{E} [e^{{\rm{T}}^{\delta} / 2}] < \infty$ holds, so there exists a probability distribution $\mathbb{Q}$ equivalent to $\mathbb{P}$ on $\mathcal{F}_{{\rm{T}}^{\delta}}$ such that $Y_{t \wedge \theta_{\eta}} = \mathbb{E}^{\mathbb{Q}}_{t \wedge \theta_{\eta}} \big[ e^{-\lambda \theta_{\eta}} Y_{\theta_{\eta}} \big]$. The process $Y$ being bounded from below by Remark \ref{bounded_below_rmk}, we have, by Fatou's Lemma, for all $t, \eta > 0$,
    \begin{equation*}
    \begin{array}{cc}
         \displaystyle Y_{t \wedge {\rm{T}}^{\delta}_{\eta}} = \liminf_{\eta' \rightarrow 0} \mathbb{E}^{\mathbb{Q}}_{t \wedge {\rm{T}}^{\delta}_{\eta}} \big[ e^{-\lambda \theta_{\eta'}} Y_{\theta_{\eta'}} \big] \geq \mathbb{E}_{t \wedge {\rm{T}}^{\delta}_{\eta}}^{\mathbb{Q}} \big[ e^{-\lambda {\rm{T}}^{\delta}} Y_{{\rm{T}}^{\delta}} \big] = +\infty \, , & \text{a.s. \hspace{-7pt} since } \, \mathbb{P}[{\rm{T}}^{\delta} < \infty] > 0  \, .
    \end{array}
    \end{equation*}

    \vspace{-5pt}
    
    \noindent Then $Y \equiv \infty$ a.s. \hspace{-7pt} on $[0, {\rm{T}}^{\delta}_{\eta}]$, which contradicts the fact that $Y \in \mathbb{S}^{\infty}_{{\rm{T}}^{\delta}_{\eta}}$.
\end{rmk}

\subsection{Well-posedness of the singular backward SDE}

We state our main existence and uniqueness results for the singular backward SDE \eqref{BSDE_LL}. Let ${\rm{T}}^{\delta}$ be the stopping time defined in \eqref{tau_eta}, and recall that $r = \frac{2-p}{p-1}$. In the rest of the paper, we consider the two constants:
\begin{equation}\label{c_0,C_0}
    \begin{array}{cc}
         \displaystyle C_i \coloneqq  \frac{(p-1)^{-\frac{2-p}{p-1}}}{2-p} \left(\frac{\sigma_i^{2}}{2 k_i \sigma_{1-i}^p} \right)^{\frac{1}{p-1}}\mathbbm{1}_{\{1 < p < 2\}} \, + \,  \frac{1}{2 k_i} \mathbbm{1}_{\{p=2\}} \, , & i \in \{0, 1\} \, .
    \end{array}
\end{equation} 

\begin{theorem}[Existence and estimates]\label{existence_thm}
    Let Assumptions {\rm{\ref{d_assumptions}}} and {\rm{\ref{driver_assumptions}}} hold. Then, 
    
    \noindent {\rm{(i)}} there exists a minimal solution $(Y, Z)$ of ${\rm{BSDE}}_{\infty}({\rm{T}}^{\delta}, f)$, and $Y$ satisfies, for all $\varepsilon > 0$,
    \vspace{-5pt}
    \begin{equation}\label{bounds_y}
    \begin{array}{cc}
         \displaystyle (C_0 - \varepsilon) \Phi_p(\delta_t) - \frac{c_{\varepsilon}}{\lambda_1} \; \leq \; Y_t \; \leq \; (C_1 + \varepsilon) \Phi_p(\delta_t) + \frac{c_{\varepsilon}}{\lambda_0} \, , & \text{for all } \, t < {\rm{T}}^{\delta} \, ,
    \end{array}
    \end{equation}
    where $c_{\varepsilon} > 0$ is a constant depending only on $p$, $\vartheta$, $k_0$, $k_1$, $\sigma_0$, $\sigma_1$, $\varepsilon$. Moreover, 
    
    \noindent {\rm{(ii)}} assume that:

    \vspace{-15pt}
    
    \begin{equation}\label{delta_BMO}
        \begin{array}{cccc}
             \displaystyle \displaystyle \Phi_p(\delta) \delta^{r + \gamma-1} \mathbbm{1}_{\{ \delta \leq R \}} \in \mathbb{H}^{{\rm{BMO}}}_{\theta} \, , &  \displaystyle\text{for some } \, R > 0,  \gamma \leq 1 \text{ and stopping time } \theta \leq {\rm{T}}^{\delta}  \,.
        \end{array}
    \end{equation}
    Then $Z \delta^{r + \gamma} \mathbbm{1}_{\{ \delta \leq R \}} \in \mathbb{H}^{{\rm{BMO}}}_{\theta}$.
\end{theorem}

The proof is reported in Section \ref{existence_prove_sec}. We recall that $\Phi_p(\delta) \delta^{r} = \mathbbm{1}_{\{ 1 <p <2 \}} - \ell(\delta) \mathbbm{1}_{\{ p=2\}}$. 

\begin{rmk}
    The BMO estimate indicates that $Z$ should formally behave like $\delta^{-(r+1)}$ when $t \to {\rm{T}}^{\delta}$ on $\{ {\rm{T}}^{\delta} < \infty \}$, since $Z \delta^{r + \gamma} = Z \delta^{r+1} \delta^{\gamma-1}$ and $\gamma-1 \leq 0$. We will see in Section \ref{markov_main_results} that the BMO estimate for $Z$ is consistent with the gradient estimate of Theorem \ref{Lasry_Lions_summary}.
\end{rmk}

The next result shows the asymptotic behavior of $Y$ near the singular time in the case $C_0 = C_1$.

\begin{corollary}\label{cor:asymptotics}
    Under the conditions of Theorem \ref{existence_thm}, assume that $\sigma_0 = \sigma_1$ if $p< 2$, and $k_0 = k_1$. Then $Y$ has the following asymptotics near ${\rm{T}}^{\delta}$:
    \begin{equation}\label{asymptotics_y}
    \begin{array}{cc}
        \displaystyle \frac{Y_t}{\Phi_p(\delta_t)} \xrightarrow[t \rightarrow {\rm{T}}^{\delta}]{} C_0 & \mathbb{P}-\text{a.s. \hspace{-7pt} on } \{ {\rm{T}}^{\delta} < \infty \} \, .
    \end{array}
    \end{equation}
\end{corollary}

\begin{proof}
    Let $\varepsilon > 0$. Then, following the estimates \eqref{bounds_y} and the fact that $C_0 = C_1$,
    \begin{equation*}
    \begin{array}{cc}
         \displaystyle (C_0 - \varepsilon) - \frac{c}{\lambda_1 \Phi_p(\delta_t)} \leq \frac{Y_t}{\Phi_p(\delta_t)} \leq (C_0 + \varepsilon) + \frac{c}{\lambda_0 \Phi_p(\delta_t)}  \, , & t \ge 0 \, ,
    \end{array}
    \end{equation*}
    so $\displaystyle C_0 - \varepsilon \leq \liminf_{t \rightarrow {\rm{T}}^{\delta}} \frac{Y_t}{\Phi_p(\delta_t)} \leq \limsup_{t \rightarrow {\rm{T}}^{\delta}} \frac{Y_t}{\Phi_p(\delta_t)} \leq C_0 + \varepsilon$ on $\{ {\rm{T}}^{\delta} < \infty \}$. We conclude by letting $\varepsilon \rightarrow 0$.
\end{proof}

We finally obtain the uniqueness of solutions of ${\rm{BSDE}}_{\infty}({\rm{T}}^{\delta}, f)$, even when $C_0 \neq C_1$, under an additional structural condition on the driver $f$.

\begin{theorem}[Uniqueness] \label{uniqueness_thm}
    Let Assumptions {\rm{\ref{d_assumptions}}} and {\rm{\ref{driver_assumptions}}} hold, and suppose that $z \in \mathbb{R}^d \mapsto f_t(y, z)$ is concave for all $y \in \mathbb{R}$, $d\mathbb{P} \otimes dt-$a.e. Then there is a unique solution of ${\rm{BSDE}}_{\infty}({\rm{T}}^{\delta}, f)$ in the sense of Definition \ref{singular_def}.
\end{theorem}

The proof is reported in Section \ref{uniqueness_control_sec}. Note that the concavity of the driver $f$ is natural from the perspective of optimal control and should be compared with the uniqueness regime in the context of quadratic backward SDEs, see \citeauthor{briand_hu2} \cite{briand_hu2}.

\subsection{Connection with optimal control}

In this section, we extend \citeauthor{Lasry_Lions_state_constraints} \cite[Theorem~VII.1]{Lasry_Lions_state_constraints} to the path-dependent setting. We assume that the generator is of the form:
\begin{equation*}
    \begin{array}{cc}
         \displaystyle f_t(y, z) = - H_t(z) - \lambda y \, , & \text{and define } \displaystyle \, g_t(a) \coloneqq \sup_{z \in \mathbb{R}^d} \, \{ a \cdot z - H_t(z) \}  \, .
    \end{array}
\end{equation*}
We show that the minimal solution $(Y, Z)$ of the BSDE \eqref{BSDE_LL} constructed in Theorem \ref{existence_thm} can be used to represent the value function of a control problem with the cost:
    \begin{equation}\label{J(alpha)}
         J(\alpha) \coloneqq \mathbb{E}^{\mathbb{P}^{\alpha}} \int_0^{\infty} e^{-\lambda t}  g_t(\alpha_t) dt \, ,
    \end{equation}
for all admissible control processes $\alpha \in \mathcal{A}$, as defined below. We adopt the weak formulation for the stochastic optimal control problem. Since we only consider drift-controlled processes, we fix the space $(\Omega, \mathcal{F})$ and the filtration $\mathbb{F}$, and work with Girsanov transformations.

\begin{definition}\label{A_def}
    Let $\mathcal{A}$ be the set of $\mathbb{F}$-progressively measurable processes $\alpha \in \bigcap_{\eta > 0} \mathbb{H}^{2, {\rm{loc}}}_{{\rm{T}}^{\delta}_{\eta}}$ for which there exists a probability distribution $\mathbb{P}^{\alpha}$ on $(\Omega, \mathcal{F})$ satisfying:
    \begin{equation*}
    \begin{array}{ccc}
        \displaystyle \frac{d \mathbb{P}^{\alpha}}{d \mathbb{P}} \Big|_{\mathcal{F}_{T \wedge {\rm{T}}^{\delta}_{\eta}}} \! \! \!= \mathcal{E} \Big( - \int_0^{\cdot} \alpha_s \cdot dW_s \Big)_{T \wedge {\rm{T}}^{\delta}_{\eta}} \, \text{ for all } T , \eta > 0 \,, & \text{and} & \mathbb{P}^{\alpha}[{\rm{T}}^{\delta} = \infty] = 1 \, .
    \end{array}
    \end{equation*}
\end{definition}

\begin{rmk}
    No absolute continuity between $\mathbb{P}^{\alpha}$ and $\mathbb{P}$ is imposed at ${\rm{T}}^{\delta}$ or on the infinite horizon: the change of measure is only locally equivalent before the levels ${\rm{T}}^{\delta}_{\eta}$. This is essential since an admissible measure satisfies $\mathbb{P}^{\alpha}[{\rm{T}}^{\delta} = \infty] = 1$, and ${\rm{T}}^{\delta}$ may be finite with positive probability under $\mathbb{P}$.
\end{rmk}

\begin{rmk}
    The probability measure $\mathbb{P}^{\alpha}$ is uniquely defined for every $\alpha \in \mathcal{A}$. Indeed, if $\mathbb{Q}^{\alpha}$ is another probability measure satisfying the conditions of Definition \ref{A_def}, then, if $T>0$ and $A \in \mathcal{F}_T$,
    \begin{equation*}
        \mathbb{P}^{\alpha} \big[ A \cap \{{\rm{T}}^{\delta}_{\eta} > T \} \big]  = \mathbb{Q}^{\alpha} \big[ A \cap \{{\rm{T}}^{\delta}_{\eta} > T \} \big] \, .
    \end{equation*}
    Taking the limit $\eta \to 0$, we conclude that $\mathbb{P}^{\alpha} = \mathbb{Q}^{\alpha}$ on $\mathcal{F}_T$, which then implies $\mathbb{P}^{\alpha} = \mathbb{Q}^{\alpha}$ on $\mathcal{F}_{\infty}$.
\end{rmk}

For the verification argument of Theorem \ref{verif_arg} below, we make some additional assumptions on the Hamiltonian $H$ to ensure that the cost $g_t$ is continuously differentiable and strongly convex.

\begin{assumption}\label{control_assumption}
    \setcounter{assumptiontag}{0}
    
    \begin{itemize}
        \item[]

        \refstepcounter{assumptiontag}
        \item[] \hspace{-30pt} (\theassumptiontag) \label{g_C^1} \hspace{1pt} The function $z \mapsto H_t(z)$ is \textit{convex} and \textit{continuously differentiable} ($C^1$).

        \refstepcounter{assumptiontag}
        \item[] \hspace{-30pt} (\theassumptiontag) \label{g_strong_convex} \hspace{1pt} $H$ is \textit{$p$-smooth} in the sense that there exists $\nu > 0$ such that:
        \begin{equation*}
            \begin{array}{cc}
                 \displaystyle H_t(z') \leq H_t(z) + \nabla H_t(z) \cdot (z'-z) + \nu |z'-z|^p \, , & \text{for all } \, z, z' \in \mathbb{R}^d \, , \; d\mathbb{P} \otimes dt- \text{a.e.}
            \end{array}
        \end{equation*} 
    \end{itemize}
\end{assumption}

By convex duality, the $p$-smoothness of $H_t$ implies the $p'$-uniform convexity of its conjugate $g_t$, where $p' = \frac{p}{p-1}$.

\begin{theorem}[Verification Argument]\label{verif_arg}
    Under Assumptions {\rm{\ref{d_assumptions}}}, {\rm{\ref{driver_assumptions}}} and {\rm{\ref{control_assumption}}}, let $(Y,Z)$ be the minimal solution of\hspace{.5pt} ${\rm{BSDE}}_{\infty}({\rm{T}}^{\delta}, f)$. Then the process defined by $\hat{\alpha} \coloneqq \nabla_{\! z} H(Z)$ satisfies:
    \begin{equation*}
        \begin{array}{ccc}
             \hat{\alpha} \in \mathcal{A} & \text{and} & \displaystyle Y_0 = J(\hat{\alpha}) =  \inf_{\alpha \in \mathcal{A}} J(\alpha) \, .
        \end{array}
    \end{equation*}
    Moreover, if $\alpha'$ is another optimal control, then $\mathbb{P}^{\alpha'} = \mathbb{P}^{\hat{\alpha}}$ and $\alpha' = \hat{\alpha}$, $d\mathbb{P} \otimes dt-$a.e. on $[0, {\rm{T}}^{\delta})$.
\end{theorem}

We prove this result in Section \ref{Control_sec}. The delicate part is to show that $\hat{\alpha}$ is an admissible process, and especially that it satisfies $ \mathbb{P}^{\hat{\alpha}}[{\rm{T}}^{\delta} = \infty] = 1$. The uniqueness of the optimal constrained law relies on the strict convexity of the criterion $J(\alpha)$ with respect to $\mathbb{P}^{\alpha}$. Note that the proof leverages the specific construction of the minimal solution $(Y, Z)$, and in particular does not require nor imply the uniqueness result of Theorem \ref{uniqueness_thm}.

\begin{example}\label{ex:control}
    If $f_t(y,z) = -\frac{1}{p}|z|^p + h(W_t) - \lambda y$, then $g_t(\alpha) = \frac{1}{p'} |\alpha_t|^{p'} + h(W_t)$ and we recover the framework of \citeauthor{Lasry_Lions_state_constraints} \cite{Lasry_Lions_state_constraints}, but with a possibly non-Markov constraint. In this case, Assumptions \ref{d_assumptions}, \ref{driver_assumptions}, \ref{control_assumption} hold if $h$ is bounded, and $\hat{\alpha}_t = \nabla_z H_t(Z_t) =  |Z_t|^{p-2} Z_t$. 
\end{example}

\subsection{The case of Markov exit time}\label{markov_main_results}

In this section, we connect our main results to the existing literature on Equation \eqref{HJB_singular}, when the state process is Markovian. The proofs are reported in Section \ref{Applications_sec}. Consider the stochastic differential equation:

\vspace{-15pt}

\begin{equation}\label{SDE_state}
\begin{array}{cc}
     \displaystyle dX_t = b(X_t) \, dt + \Sigma(X_t) \, dW_t \, , & X_0 = x \in D \, ,
\end{array}
\end{equation}
where the coefficients $b : \mathbb{R}^d \rightarrow \mathbb{R}^d$, $\Sigma : \mathbb{R}^d \rightarrow \mathbb{R}^{d \times d}$ and the domain $D$ satisfy Assumption \ref{state_process_assumptions} below. 

\begin{assumption} \label{state_process_assumptions}
\setcounter{assumptiontag}{0}

    \begin{itemize}
        \item[]

    \refstepcounter{assumptiontag}
        \item[] \hspace{-30pt} (\theassumptiontag) \label{D_condition} \hspace{1pt} $D$ is an open subset of $\mathbb{R}^d$ and there exists $\overline{\eta} > 0$ such that $d_{\pm} \in C^2(D^{\overline{\eta}})$.

        \refstepcounter{assumptiontag}
        \item[] \hspace{-30pt} (\theassumptiontag) \label{Lpz_condition_b_sigma} \hspace{1pt} There exists a constant $L > 0$ such that:
        \begin{equation*}
            \begin{array}{cc}
                 \displaystyle |b(x) - b(x')| + | \Sigma(x) - \Sigma(x') | \leq L |x - x'| \, , & \text{for all } x, x' \in D \, .
            \end{array}
        \end{equation*}

        \refstepcounter{assumptiontag}
        \item[] \hspace{-30pt} (\theassumptiontag) \label{uniform_ellipticity} \hspace{.5pt} There exist $0 < \sigma_0 \leq \sigma_1 $ such that:
        \begin{equation*}
        \begin{array}{cc}
            \displaystyle \sigma_0^2 |\xi|^2 \leq \xi^{\intercal} \, \Sigma \Sigma^{\intercal}(x) \,\xi \leq \sigma_1^2 |\xi|^2 \, , & \text{for all } x, \xi \in D.
        \end{array} 
        \end{equation*}
        
    \end{itemize}
\end{assumption}

Under Assumption \eqref{Lpz_condition_b_sigma}, there exists a unique strong $\mathbb{F}$-adapted Markov process $X$ solving \eqref{SDE_state}. We apply our framework to the case:
\begin{equation*}
    \rm{T}^{\delta} = \inf \{ t \geq 0 : X_t \notin D \} = \inf \{ t \ge 0 : \mathbf{d}(X_t) \leq 0 \} \, ,
\end{equation*}
where $\mathbf{d}$ is a truncated and regularized distance to the boundary $\partial D$, see \eqref{regularized_distance}. 

In the following we consider a general non-Markov driver $f$ satisfying Assumption \ref{driver_assumptions}.
We shall derive some additional integrability and estimates on the process $Z$ of the singular backward SDE \eqref{BSDE_LL}, which are consistent with the results of \citeauthor{Lasry_Lions_state_constraints} \cite{Lasry_Lions_state_constraints}. Moreover, we prove the connection with the PDE \eqref{HJB_singular} with boundary blow-up when the driver is Markovian.

\begin{theorem}\label{thm:Markov_BSDE}
    Under Assumption {\rm{\ref{state_process_assumptions}}}, there exists a minimal solution $(Y, Z)$ of\hspace{.5pt} ${\rm{BSDE}}_{\infty}({\rm{T}}^{\delta}, f)$, where $Y$ satisfies the estimates \eqref{bounds_y}, with $\delta_t = \mathbf{d}(X_t)$. Moreover, if $D' \subset D$ is a bounded subdomain with $C^2$-regular boundary, then \eqref{delta_BMO} holds true for all $\gamma > \frac{1}{2}$ with $\theta \coloneqq \inf \{ t \ge 0 : X_t \notin D' \}$, and consequently $Z \, \delta^{r + \gamma} \in \mathbb{H}^{\textnormal{BMO}}_{\theta}$.
\end{theorem}

\begin{rmk}
     In particular, if the domain $D$ is bounded, we obtain $Z \, \delta^{r + \gamma} \in \mathbb{H}^{\textnormal{BMO}}_{{\rm{T}}^{\delta}}$ for all $\gamma > \frac{1}{2}$, then $Z \in \mathbb{H}^{\textnormal{BMO}}_{{\rm{T}}^{\delta}_{\eta}}$ for all $\eta > 0$. This last point can also be seen as a consequence of Remark \ref{ST_condition} and Lemma \ref{BMO_lemma} below.
\end{rmk}

\begin{rmk}\label{ST_condition}
   In the case $p<2$ and $\gamma = 1$, condition \eqref{delta_BMO} together with $\delta \in \mathbb{S}^{\infty}_{\theta}$ reduces to:
\begin{equation}\label{theta_condition}
    \sup_{\kappa \in \mathcal{T}_0^{\theta}} \big\| \mathbb{E}[\theta - \kappa]  \big\|_{\infty}  < \infty \, .
\end{equation}
    This condition holds true when $\theta$ is the first exit time of $X$ from a bounded domain $D$: a standard consequence of the strong Markov property is
    \begin{equation*}
        \sup_{\kappa \in \mathcal{T}_0^{\theta}} \| \mathbb{E}_{\kappa} [\theta - \kappa] \|_{\infty} \leq \sup_{x' \in D} \mathbb{E}_{x'} [\theta] < \infty \, .
    \end{equation*}
    This estimate is classical for uniformly elliptic diffusion processes in bounded domains, see \textit{e.g.} \citeauthor{pinsky_book} \cite[Section 2, Theorem~2.1]{pinsky_book}. A full characterization of this condition is available in the early work of \citeauthor{exit_time_BMO} \cite{exit_time_BMO} when $\theta$ is the exit time of $X$ from a general open domain of $\mathbb{R}^d$.
\end{rmk}

We finally connect the singular backward SDE \eqref{BSDE_LL} to the PDE \eqref{HJB_singular} with boundary blow-up of \citeauthor{Lasry_Lions_state_constraints}. This type of result is classical in the BSDE literature, see \citeauthor{kobylanski} \cite{kobylanski} for quadratic BSDEs and boundary value problems.

\begin{proposition}\label{connection_PDE_prop}
    Assume that $D$ is a bounded domain with $C^2$-regular boundary, and let $u$ be the unique solution of \eqref{HJB_singular} in $D$ with boundary blow-up given in Theorem \ref{Lasry_Lions_summary}. Let $X$ be the process defined in \eqref{SDE_state} with $b \equiv 0$ and $\Sigma \equiv I_d$, and define the processes:
    \begin{equation*}
        \begin{array}{cccc}
             Y_t \coloneqq u(X_t) , & Z_t \coloneqq \nabla u(X_t) , & \delta_t \coloneqq \mathbf{d}(X_t) , & \displaystyle f_t(y, z) \coloneqq - \frac{1}{p}|z|^p + h(X_t) - \lambda y \, .
        \end{array}
    \end{equation*}
    Then $(Y, Z)$ is the unique solution of\hspace{.5pt} ${\rm{BSDE}}_{\infty}({\rm{T}}^{\delta}, f)$, and $Z \, \mathbf{d}(X)^{r + \gamma} \in \mathbb{H}^{\textnormal{BMO}}_{{\rm{T}^{\delta}}}$ for all $\gamma > \frac{1}{2}$.
\end{proposition}

\begin{rmk}\label{ergodic_rmk}
    Denote by $u_{\lambda}$ the unique solution of the PDE \eqref{HJB_singular} with discount $\lambda > 0$, where $D$ is bounded with a $C^2$-regular boundary.
    In the Markovian setting, the discounted problem is closely related to the ergodic singular equation obtained as $\lambda \to 0^+$. \citeauthor{Lasry_Lions_state_constraints} proved in \cite[Theorem~VI.1]{Lasry_Lions_state_constraints}, under suitable assumptions, the local convergence of $u_{\lambda}-u_{\lambda}(x_0)$ and of $\lambda u_{\lambda}$ towards a solution of the corresponding ergodic problem, for any $x_0 \in D$. A crucial ingredient in their argument is a gradient estimate which is uniform with respect to $\lambda$. Our weighted BMO estimates are consistent with their gradient bound $\|\nabla u_{\lambda} \, \mathbf{d}^{r +1}\|_{\infty} < \infty$ for each $\lambda > 0$, but we do not obtain the uniform estimate required to pass to the ergodic limit for a general non-Markovian generator and constraint, as their argument requires estimates on the Hessian of $u$, see \cite[Appendix]{Lasry_Lions_state_constraints}, which we are unfortunately unable to obtain in the current non-Markovian setting. It is natural to ask whether the optimal measures $\mathbb{P}^{\hat{\alpha}_{\lambda}}$ converge towards an optimal law for the ergodic state-constrained problem.
\end{rmk}

\begin{rmk}
    In the quadratic case $p=2$, the ergodic Markovian problem admits the familiar connection with conditioned diffusions. If $h \equiv 0$, the exponential transformation of the ergodic solution $\psi = e^{-u}$ leads to a positive eigenfunction of the killed diffusion generator, and the corresponding optimal drift coincides with the drift of the associated Doob $h$-transform:
    \begin{equation*}\label{drift_q_process}
    \hat{\alpha}(x) \, = \nabla u(x) \, = \, \nabla (-\log \psi) (x) \, = \, - \frac{\nabla \psi}{\psi}(x) \, ,
\end{equation*}
    This identifies the optimally controlled process with the $Q$-process, \textit{i.e.} the diffusion conditioned on long-term survival. 
\end{rmk}

\begin{example}
    When $d=1$ and $D = \mathbb{R}_+^*$, the BMO estimate of Theorem \ref{thm:Markov_BSDE} and condition \eqref{delta_BMO} indicate that $Z_t$ behaves like a multiple of $X_t^{-(r+1)} = X_t^{-\frac{1}{p-1}}$ when $t \to {\rm{T}}^{\delta}$ on $\{ {\rm{T}}^{\delta} < \infty\}$.
    This is consistent with the explicit solutions that one can exhibit in the limit case $\lambda \to 0^+$, see \citeauthor{porretta_veron} \cite[Theorem~4.1(ii)]{porretta_veron}. Indeed, if $f_t(y,z) = -\frac{1}{p} |z|^p$, then $u_{\lambda} - u_{\lambda}(x_0) \to v$ and $\lambda u_{\lambda} \to 0$ locally, where $v$ satisfies:
\begin{equation*}
    \begin{array}{ccc}
        \displaystyle - \frac{1}{2} v'' + \frac{1}{p}|v'|^p = 0 \,, & \text{and} & v(x) \xrightarrow[x \to 0]{} \infty \, .
    \end{array}
\end{equation*}
This equation has an explicit solution, with $v$ being unique up to an additive constant:
\begin{equation*}
\begin{array}{cccc}
   \hspace{-5pt} \displaystyle v(x) = \frac{C_0}{x^{\frac{2-p}{p-1}}} \mathbbm{1}_{\{ 1<p<2\}} - \log(x) \mathbbm{1}_{\{ p=2\}} & \text{and} & \displaystyle v'(x) = - \big( \frac{2-p}{p-1} \big) \frac{C_0}{x^{\frac{1}{p-1}}} \mathbbm{1}_{\{ 1<p<2\}} - \frac{1}{x} \mathbbm{1}_{\{ p=2\}}  \, .
\end{array}
\end{equation*}
We can also compute the optimal control process $\hat{\alpha}(x) = -|v'|^{p-1} $ explicitly: 
\begin{equation*}
    \begin{array}{cc}
         \displaystyle \hat{\alpha}(x) = - \big( \frac{2-p}{p-1}\big)^{p-1} \frac{C_0^{p-1}}{x} \mathbbm{1}_{\{ 1<p<2\}} - \frac{1}{x} \mathbbm{1}_{\{ p=2\}} = - \frac{p'}{2 x} \, , & \text{where } \displaystyle p' = \frac{p}{p-1}
    \end{array}
\end{equation*}
The optimal distribution $\mathbb{P}^{\hat{\alpha}}$ corresponds to the $p'+1$-dimensional Bessel process $dX_t = \frac{p'}{2X_t} dt + dW_t^{\hat{\alpha}}$. In particular when $p =p'=2$, $\mathbb{P}^{\hat{\alpha}}$ is the $Q$-process of the Brownian motion in $\mathbb{R}_+^*$. 
\end{example}

\section{Estimates and well-posedness} \label{existence_sec}

In this section, our objective is to prove Theorem \ref{existence_thm} by adapting the arguments of \citeauthor{Lasry_Lions_state_constraints} \cite{Lasry_Lions_state_constraints} to the backward SDE setting.

\subsection{Bounded mean oscillations and comparison principle}\label{BMO_comp_principle_sec}

We start by stating an important integrability property of solutions to the backward SDE \eqref{BSDE_LL}. Then we prove a comparison result, which is essential in the rest of the paper. These results are derived in a general setting, in particular the driver $f$ is only required to satisfy the conditions \eqrange{f_prog_meas}{f_monotonicity} and not Assumption \eqref{p_growth}.

\vspace{5pt}

\begin{lemma}\label{BMO_lemma}
    Let $\theta$ be a $\mathbb{F}$-stopping time and $\xi$ be a bounded $\mathcal{F}_{\theta}$-measurable random variable. Then there exists a unique solution $(Y, Z)$ of\hspace{.5pt} ${\rm{BSDE}}_{\xi}(\theta, f)$ such that:
    \begin{equation}\label{sigma_BMO}
        \| Z \|^2_{{\rm{BMO}}, \theta} \leq K \Big( 1 + \sup_{\kappa \in \mathcal{T}_0^{\theta}} \big \| \,  \mathbb{E}_{\kappa} \big[ \theta - \kappa \big] \,  \big \|_{\infty} \Big) \, ,
    \end{equation}
    for some constant $K > 0$ depending on $C_f$ and $\| Y \|_{\infty}$. 
\end{lemma}
\begin{proof}
    The existence and uniqueness of the solution were established in \citeauthor{QBSDE_random_time} \cite[Theorem~3.3]{QBSDE_random_time}. To derive the BMO estimate we use a standard argument in quadratic BSDEs, see \textit{e.g.} \citeauthor{zhang_BSDE} \cite[Theorem~7.2.1]{zhang_BSDE}. Let $T> 0$, $\kappa \in \mathcal{T}_0^{\theta}$ and $\gamma > 0$. Applying Itô's formula to $\varphi(Y_{\cdot})$, where $\varphi(y) = e^{\gamma y}$, we obtain:
    \begin{align}
        e^{\gamma Y_{T \wedge \kappa}} &= e^{\gamma Y_{T \wedge \theta}} +  \int_{T \wedge \kappa}^{T\wedge \theta} \gamma e^{\gamma Y_s} \, \big[ f_s(Y_s, Z_s) - \frac{1}{2} \gamma |Z_s|^2 \big] ds + \int_{T \wedge \kappa}^{T\wedge \theta}  \gamma e^{\gamma Y_s} \, Z_s \cdot dW_s \nonumber \\[5pt]
        &\leq e^{\gamma Y_{T \wedge \theta}} + \int_{T \wedge \kappa}^{T\wedge \theta} \gamma e^{\gamma Y_s} \,  \big[ C_f(1 + |Y_s| + |Z_s|^2) - \frac{1}{2} \gamma |Z_s|^2 \big] ds + \int_{T \wedge \kappa}^{T\wedge \theta} \gamma e^{\gamma Y_s}  \, Z_s \cdot dW_s \nonumber \\[5pt]
        &\leq e^{\gamma Y_{T \wedge \theta}} + \int_{T \wedge \kappa}^{T\wedge \theta} \gamma e^{\gamma Y_s}  \, \big[ -|Z_s|^2 +  \gamma \, C_f \, (1 + |Y_s|) \big] ds + \int_{T \wedge \kappa}^{T\wedge \theta} \gamma e^{\gamma Y_s} \,  Z_s \cdot dW_s \, , \nonumber
    \end{align}
    by choosing $\gamma = 2(1 + C_f)$. As $(Y, Z) \in \mathbb{S}^{\infty}_{\theta} \times \mathbb{H}^2_{T \wedge \theta}$, the stochastic integral is a square integrable martingale, and we obtain some constant $K$ depending on $C_f$ and $\| Y \|_{\infty}$ such that:
    \begin{align}
        \mathbb{E}_{\kappa} \int_{\kappa}^{\theta} |Z_s|^2 ds  \, = \,  \lim_{T \rightarrow \infty} \mathbb{E}_{\kappa} \int_{T \wedge \kappa}^{T\wedge \theta} |Z_s|^2 ds \,  & \leq \,  K\big( 1 + \, \mathbb{E}_{\kappa}\big[ (T \wedge\theta) - (T \wedge \kappa) \big] \big) \nonumber \\
        &\leq K \Big( 1 + \sup_{\kappa \in \mathcal{T}_0^{\theta}} \big \| \,  \mathbb{E}_{\kappa} \big[ \theta - \kappa \big] \,  \big \|_{\infty} \Big) \nonumber \, .
    \end{align}
    
    \vspace{-10pt}
    
\end{proof}

\begin{rmk}\label{rmk:BMO}
    The notion of bounded mean oscillations is central in the study of quadratic backward SDEs. We recall the following crucial property, which can be found in \citeauthor{Kazamaki_BMO} \cite[Theorem~2.3]{Kazamaki_BMO}: if $\psi \in \mathbb{H}^{\text{BMO}}_{\theta}$, then $\psi \in \mathbb{H}^k_{\theta}$ for all $k \geq 1$ and the stochastic exponential $\zeta \coloneqq  \mathcal{E}\big( \int_0^{\cdot} \psi_s dW_s \big)$ is a uniformly integrable martingale on $[0, \theta]$.
\end{rmk}

\begin{proof}[Proof of Lemma \ref{comp_principle}]
    ($\rm{\ast}$) We first assume that $\xi'$ is bounded. Denote $\Delta Y = Y - Y'$, $\Delta Z = Z-Z'$, $\Delta f = f-f'$, $\Delta \xi = \xi - \xi'$ and $\Delta K = K - K'$, where $K$ and $K'$ are given in Definition \ref{singular_def}. Let $0 \leq t < T < \infty$. Applying Itô-Tanaka's formula to $e^{-\lambda_0 s} \big(\Delta Y_s)^+$, we have:
    \begin{align}
        e^{-\lambda_0 (t\wedge \theta)}\big(\Delta Y_{t \wedge\theta}\big)^+ =  e&^{-\lambda_0 (T\wedge \theta)}\big(\Delta Y_{T \wedge \theta}\big)^+  +\int_{t\wedge \theta}^{T\wedge \theta}e^{-\lambda_0 s} \, \mathbbm{1}_{\{\Delta Y_s \geq 0\}} \,   \Delta f_s(Y_s, Z_s)  \, ds \nonumber \\[6pt]
        & + \int_{t\wedge \theta}^{T\wedge \theta}e^{-\lambda_0 s} \, \mathbbm{1}_{\{\Delta Y_s \geq 0\}} \, \Big[ \big(f'_s(Y_s, Z_s) - f'_s(Y'_s, Z_s) \big) + \big(f'_s(Y'_s, Z_s) - f'_s(Y'_s, Z'_s) \big) \Big] \, ds \nonumber \\[6pt]
        & -  \int_{t\wedge \theta}^{T\wedge \theta} e^{-\lambda_0 s} \,  \mathbbm{1}_{\{\Delta Y_s \geq 0\}} \, \Delta Z_s \cdot d W_s - \int_{t \wedge \theta}^{T \wedge \theta} e^{-\lambda_0 s} \,  \mathbbm{1}_{\{\Delta Y_s \geq 0\}} \, ( d (\Delta K_s) + d L^0_s ) \, , \nonumber
    \end{align}    
    where $L^0$ is the local time of $\Delta Y$ in 0. A first simplification comes from the fact that $\Delta f_t(Y_t, Z_t) \leq 0$, $d\mathbb{P} \otimes dt-$a.e. Moreover, $\mathbbm{1}_{\{\Delta Y_s \geq 0\}} \big(f'_s(Y_s, Z_s) - f'_s(Y_s', Z_s) \big) \leq - \lambda_0  (\Delta Y_s)^+ \mathbbm{1}_{\{\Delta Y_s \geq 0\}} \leq 0$ by Assumption \eqref{f_monotonicity}. Finally, by property of the local time and the monotonicity of $\Delta K$, the last term is non-positive. Therefore, we obtain the following inequality:
    \begin{equation*}
        e^{-\lambda_0 (t\wedge \theta)}\big(\Delta Y_{t \wedge\theta}\big)^+ \leq  e^{-\lambda_0 (T\wedge \theta)}\big(\Delta Y_{T \wedge \theta}\big)^+ -  \int_{t\wedge \theta}^{T\wedge \theta} e^{-\lambda_0 s} \,  \mathbbm{1}_{\{\Delta Y_s \geq 0\}} \, \Delta Z_s \cdot \big(d W_s - \beta_s ds \big)  \, ,
    \end{equation*}
    where $\beta_s \coloneqq \frac{f'_s(Y_s', Z_s) - f'_s(Y_s', Z_s')}{|Z_s - Z_s'|^2} \mathbbm{1}_{\{Z_s \neq Z_s'\}} (Z_s - Z_s')$. By Assumption \eqref{f_prog_meas} and Lemma \ref{BMO_lemma}, we deduce that $\beta \in \mathbb{H}^{\text{BMO}}_{T \wedge\theta}$. By Girsanov's Theorem, we obtain a new measure $\mathbb{Q}_{T}$ with density $\frac{d \mathbb{Q}_T}{d \mathbb{P}} = \mathcal{E}\big( \int_0^{\cdot} \psi_s dW_s \big)$ such that $W_{\cdot} - \int_0^{\cdot} \mathbbm{1}_{s \leq \theta} \beta_s ds$ is a $\mathbb{Q}_T$-Brownian motion on $[0, T]$. It follows from \citeauthor{zhang_BSDE} \cite[Theorem~7.2.3]{zhang_BSDE} that $\Delta Z$ belongs to $\mathbb{H}^2_{T \wedge \theta}( \mathbb{Q}_T)$, so taking the expectation with respect to $\mathbb{Q}_T$ leads to:
    \begin{align}\label{sigma_split_two_cases}
        e^{-\lambda_0 (t \wedge \theta)} (\Delta Y_{t \wedge \theta})^+ &\leq  \mathbb{E}^{\mathbb{Q}_T}_t \Big[ e^{-\lambda_0 (T \wedge \theta)} (\Delta Y_{T \wedge \theta} )^+ \Big] \nonumber  \\[3pt]
        & = \mathbb{E}^{\mathbb{Q}_T}_t \Big[ e^{-\lambda_0 \theta} (\Delta \xi )^+ \mathbbm{1}_{\{ \theta \leq T\}}  \Big] + \mathbb{E}^{\mathbb{Q}_T}_t \Big[ e^{-\lambda_0 T} (\Delta Y_{T})^+ \mathbbm{1}_{\{ \theta > T\}} \Big] \nonumber \\[3pt]
        &\leq e^{-\lambda_0 T} \| \Delta Y \|_{\infty} \, \xrightarrow[T \to \infty]{}0  \nonumber \, ,
    \end{align}
    since $\Delta \xi \leq 0$ a.s. \hspace{-7pt} and $Y, Y' \in \mathbb{S}^{\infty}_{\theta}$. This proves that $\Delta Y_t \leq 0$ a.s. \hspace{-7pt} for all $t \leq \theta$. If $Y_0 = Y'_0$, the same arguments with $\big( \Delta Y \big)^-$ show that $Y_T = Y_T'$ for all $T \geq 0$. \\

    \vspace{-8pt}
    
    ($\rm{\ast \ast}$) Now suppose that $\theta = {\rm{T}}^{\delta}$ and $\xi' \equiv +\infty$. Define the stopping time $\theta_N \coloneqq \inf \big\{ t \geq 0: Y'_t \geq N \big \}$. Taking $N \geq \|\xi\|_{\infty}$ and noticing that $Y' \in \mathbb{S}^{\infty}_{\theta_N}$, we may apply ($\rm{\ast}$) to $\theta_N$ and obtain $Y_t \leq Y_t'$ for all $t \leq \theta_N$. As $\theta_N \rightarrow \theta$ a.s. \hspace{-7pt} by Remark \ref{stopping_time_def_rmk}, this implies that $Y \leq Y'$ on $[0, \theta]$.
\end{proof}

\subsection{Adapting Lasry \& Lions' arguments to the non-Markovian framework}

Let $\varepsilon > 0$ and $\eta \in [0, \overline{\eta})$, where $\overline{\eta}$ was introduced in Assumption \ref{d_assumptions}. Following \citeauthor{Lasry_Lions_state_constraints} \cite{Lasry_Lions_state_constraints}, we define the processes:
\begin{equation*}
    \begin{array}{ccc}
      \displaystyle \overline{w}^{\varepsilon, \eta}_t \coloneqq (C_{1} + \varepsilon) \Phi_p(\delta_t - \eta) + \frac{c}{\lambda_0}  & \text{and} &  \displaystyle \underline{w}^{\varepsilon, \eta}_t \coloneqq (C_0 - \varepsilon) \Phi_p(\delta_t + \eta) - \frac{c}{\lambda_1} \, ,
\end{array}
\end{equation*}
where the constants $C_i$ are defined in \eqref{c_0,C_0} and $\Phi_p$ is the function \eqref{Phi_p}. Notice that the process $\overline{w}^{\varepsilon, \eta}$ blows up slightly before the exit time ${\rm{T}}^{\delta}$ so it will play the role of supersolution, while $\underline{w}^{\varepsilon, \eta}$ blows up after ${\rm{T}}^{\delta}$ so we will use it as a subsolution. The positive parameter $c$ is an adjustable constant.

\begin{proposition}\label{lasry_lions_estimates}
    Under Assumption \ref{driver_assumptions}, there exist two $\mathbb{R}^d$-valued random processes $\overline{z}^{\varepsilon, \eta}$ and $\underline{z}^{\varepsilon, \eta}$ such that  $(\overline{w}^{\varepsilon, \eta}, \overline{z}^{\varepsilon, \eta})$ (resp. $(\underline{w}^{\varepsilon, \eta}, \underline{z}^{\varepsilon, \eta})$) is a supersolution (resp. subsolution) of\hspace{.5pt} ${\rm{BSDE}}_{\infty}({\rm{T}}^{\delta}, f)$.
\end{proposition}

\begin{proof} 
(i) We start by studying the process $\overline{w}^{\varepsilon, \eta}$. Denote $C_{\varepsilon} \coloneqq C_1 + \varepsilon$. A direct application of Itô's formula leads to $\displaystyle d\overline{w}^{\varepsilon, \eta}_t = \overline{q}^{\varepsilon, \eta}_t \, dt + \overline{z}^{\varepsilon, \eta}_t \cdot dW_t$, where:
\begin{equation*}
    \begin{array}{ccc}
        \displaystyle \overline{q}^{\varepsilon, \eta}_t \coloneqq \frac{1}{2} \big|U_t \big|^2 C_{\varepsilon}\Phi_p''(\delta_t - \eta) + C_{\varepsilon}\Phi'_p(\delta_t - \eta) \, \Xi_t & \text{and} & \displaystyle \overline{z}^{\varepsilon, \eta}_t \coloneqq C_{\varepsilon} \Phi'_p(\delta_t - \eta) \, U_t  \, .
    \end{array}
\end{equation*}
First, note that $\overline{w}^{\varepsilon, \eta}$ and $\overline{z}^{\varepsilon, \eta}$ are bounded on each event $\{ t \leq {\rm{T}}^{\delta}_{\eta'} \}$ when $\eta' > \eta_{\text{sup}} \coloneqq \eta$, so in particular $(\overline{w}^{\varepsilon, \eta},\overline{z}^{\varepsilon, \eta}) \in \mathbb{S}^{\infty}_{{\rm{T}}^{\delta}_{\eta'}} \times \mathbb{H}^2_{{\rm{T}}^{\delta}_{\eta'}}$. We also have $\lim_{\eta' \rightarrow \eta} \overline{w}^{\varepsilon, \eta}_{{\rm{T}}^{\delta}_{\eta'}} = +\infty$ a.s. on $\{ {\rm{T}}^{\delta} < \infty\}$. Therefore, we only need to check that $\dot{K}_t \coloneqq -\overline{q}^{\varepsilon, \eta}_t - f_t (\overline{w}^{\varepsilon, \eta}_t, \overline{z}^{\varepsilon, \eta}_t ) \geq 0$, $d\mathbb{P}\otimes dt-$a.e. \hspace{-7pt} for some constant $c > 0$. \\

\vspace{-8pt}

\noindent \underline{\textit{Case $1 < p < 2$:}} Plugging the expression of $\Phi_p$ into the above formula and using Assumption \ref{driver_assumptions}, we obtain:
    \begin{align}\label{estimates_ineq_proof}
        \dot{K}_t &= -\frac{r(r+1) C_{\varepsilon}}{2 ( \delta_t - \eta)^{r+2}} \big| U_t \big|^2 + \frac{r  C_{\varepsilon}}{( \delta_t - \eta )^{r+1}} \Xi_t - f_t\Big( \frac{C_{\varepsilon}}{(\delta_t - \eta )^r} +  \frac{c}{\lambda_0} \, , \,  - \frac{r C_{\varepsilon}}{( \delta_t - \eta )^{r+1}} U_t \Big) \nonumber \\[6pt]
        & \geq -\frac{r(r+1)\, C_{\varepsilon}}{2 ( \delta_t - \eta)^{r+2}} \sigma_1^2 - \frac{r  C_{\varepsilon} L_{\Xi} (1 + |\delta_t|^m)}{( \delta_t - \eta )^{r+1}}  + k_1 \frac{ r^p C_{\varepsilon}^p}{( \delta_t - \eta )^{p(r+1)}} |U_t|^p - \vartheta + \lambda_0 \frac{C_{\varepsilon}}{(\delta_t - \eta )^r} + c \, . \nonumber
    \end{align}
    Here we used the monotonicity of $f$ in $y$ of Assumption \eqref{f_monotonicity} and the growth condition in $z$ of Assumption \eqref{p_growth}. Note that $r$ satisfies $r+2 = p(r+1)$. In the case $\delta_t \geq \overline{\eta} > \eta$, there exists a constant $M > 0$ depending on $\eta$, $\overline{\eta}$ and the parameters of the problem such that $\dot{K}_t \geq -M + c$. On the other hand, if $\delta_t \in (\eta, \overline{\eta}]$, we conclude by Assumption \ref{d_assumptions} that:
    \begin{align}
        \dot{K}_t \geq \frac{1}{( \delta_t - \eta )^{r+2}} \Big\{ -\frac{ 1 }{2}&r(r+1) \sigma_1^2 C_{\varepsilon} + k_1 r^p \sigma_0^{p} C_{\varepsilon}^p \Big\} - \frac{r  C_{\varepsilon} L_{\Xi} (1 + |\delta_t|^m)}{( \delta_t - \eta )^{r+1}} + \lambda_0 \frac{C_{\varepsilon}}{(\delta_t - \eta )^r} + c - \vartheta \, . \nonumber
    \end{align}
    The numerator of the first term can be split into two:
    \begin{equation}\label{justification_C_0}
        \Big\{ -\frac{ 1 }{2} r(r+1)  C_1 \sigma_1^2 + k_1 r^p C^p_1 \sigma_0^{p} \Big\} + \Big\{ -\frac{1}{2} r(r+1) \sigma_1^2 \varepsilon + k_1 r^p  \sigma_0^{p} C_1^{p} \Big[ \Big( 1 + \frac{\varepsilon}{C_1} \Big)^p - 1 \Big] \Big\} \, .
    \end{equation}
    By choosing $C_1$ as in \eqref{c_0,C_0}, the first term of \eqref{justification_C_0} vanishes and there is a constant $\nu > 0$ such that:
    \begin{equation*}
        \dot{K}_t \geq \frac{\nu \, \varepsilon}{( \delta_t - \eta )^{r+2}} - \frac{r  C_{\varepsilon} L_{\Xi} (1 + |\delta_t|^m)}{( \delta_t - \eta )^{r+1}} + \lambda_0 \frac{C_{\varepsilon}}{(\delta_t - \eta )^r} + c - \vartheta \, .
    \end{equation*}
    if $\varepsilon > 0$ is small enough. Since $m \leq \frac{1}{p-1} = r+1$, we can see that $\dot{K}_t$ is almost surely bounded from below. Adjusting the constant $c$ if necessary, we obtain $\dot{K}_t \geq 0$ a.s.  \\

    \vspace{-8pt}

    \noindent \underline{Case $p=2$:} We adapt the argument of \citeauthor{Lasry_Lions_state_constraints} \cite{Lasry_Lions_state_constraints} with the slight difference that we need to consider the truncation $\ell$ of the logarithm function defined in \eqref{smooth_log_truncation} to address the possible unboundedness of $\delta$. Following the same steps as above, we obtain:
    \begin{align}
        \dot{K}_t &\geq -\frac{C_{\varepsilon}}{2 ( \delta_t - )^2} \big| U_t \big|^2 + \frac{C_{\varepsilon}}{ \delta_t - \eta }\Xi_t - f_t \Big( -C_{\varepsilon} \, \ell ( \delta_t - \eta ) + c \,  , \,  - \frac{C_{\varepsilon}}{\delta_t - \eta } U_t \Big) - C_{\varepsilon} L_{\log} \nonumber \\[6pt]
        &\geq \frac{\big| U_t \big|^2}{( \delta_t - \eta\big)^{2}} \Big\{ -\frac{ 1 }{2} C_{\varepsilon} + k  \,C_{\varepsilon}^2 \Big\} - \frac{C_{\varepsilon} L_{\Xi} (1 + |\delta_t|)}{\delta_t - \eta} + \lambda_0 C_{\varepsilon}  \, \ell (\delta_t - \eta ) + \lambda_0 c - \vartheta - C_{\varepsilon} L_{\log} \, . \nonumber \label{supersol_p=2}
    \end{align}
    Then, using the expression of $C_1$ in \eqref{c_0,C_0} we can choose the parameter $c>0$ so that this is nonnegative. \\

    \vspace{-8pt}

    (ii) We now establish that $\underline{w}^{\varepsilon, \eta}$ is a subsolution. Denote $c_{\varepsilon} \coloneqq C_0 - \varepsilon$. Once again, Itô's formula gives $\displaystyle d \underline{w}^{\varepsilon, \eta}_t = \underline{q}^{\varepsilon, \eta}_t \, dt + \underline{z}^{\varepsilon, \eta}_t \cdot dW_t$, where:
    \begin{equation*}
    \begin{array}{ccc}
        \displaystyle \underline{q}^{\varepsilon, \eta}_t \coloneqq \frac{1}{2} \big|U_t \big|^2 C_{\varepsilon}\Phi_p''(\delta_t + \eta) + C_{\varepsilon}\Phi'_p(\delta_t + \eta) \, \Xi_t & \text{and} & \displaystyle \underline{z}^{\varepsilon, \eta}_t \coloneqq C_{\varepsilon} \Phi'_p(\delta_t + \eta) \, U_t \, .
    \end{array}
    \end{equation*}
    It is clear that the integrability conditions of Definition \ref{singular_def} hold with $\eta_{\text{sub}} \coloneqq -\eta$, so we only need to prove that $\dot{K}_t \coloneqq - \underline{q}^{\varepsilon, \eta}_t - f_t( \underline{w}^{\varepsilon, \eta}_t, \underline{z}^{\varepsilon, \eta}_t) \leq 0$. \\

    \vspace{-5pt}
    
    \noindent \underline{Case $1 < p < 2$:} If $\delta_t \in [-\overline{\eta}, -\eta)$, we have, by the same arguments as above and Assumption \ref{driver_assumptions}:
    \begin{align}
        \dot{K}_t & \leq -\frac{r(r+1)c_{\varepsilon}}{2 ( \delta_t + \eta)^{r+2}} \sigma_0^2 + \frac{r c_{\varepsilon}}{( \delta_t + \eta )^{r+1}} | \Xi_t | + k_0 \frac{r c_{\varepsilon}}{( \delta_t +\eta )^{p(r+1)}} \sigma_1^p - \lambda_1 \Big| \frac{c_{\varepsilon}}{(\delta_t + \eta )^r} - \frac{c}{\lambda_1} \Big| + \vartheta \nonumber \\[6pt]
        & \leq -\frac{\nu \, \varepsilon}{( \delta_t - \eta )^{r+2}} + \frac{r  c_{\varepsilon} L_{\Xi} (1 + |\delta_t|^m)}{( \delta_t - \eta )^{r+1}} + \lambda_1 \frac{c_{\varepsilon}}{(\delta_t - \eta )^r} + c - \vartheta \nonumber \, ,
    \end{align}
    for some $\nu > 0$. We conclude that $\dot{K}_t \leq 0$ given the choice of $C_0$ and after adjusting the constant $c$. \\

    \vspace{-8pt}

    \noindent \underline{Case $p=2$:} Similarly, we can write:
    \begin{equation*}\label{subsol_p=2}
        \dot{K}_t \leq \frac{\big| U_t \big|^2}{\big( \delta_t + \eta\big)^{2}} \Big\{ -\frac{ 1 }{2} c_{\varepsilon} + K  \,c_{\varepsilon}^2 \Big\} + \frac{c_{\varepsilon} L_{\Xi} (1 + |\delta_t|)}{ \delta_t - \eta} + \lambda_1 \Big| c_{\varepsilon} \, \ell \big(\delta_t + \eta \big) + \frac{c}{\lambda_1} \Big| + C_{\varepsilon} L_{\log} \, .
    \end{equation*}
    Similarly to the previous cases, we can adjust the constant $c$ to obtian $\dot{K}_t \leq 0$.
\end{proof}

\subsection{Existence and estimates}\label{existence_prove_sec}

We now turn to the core proof of Theorem \ref{existence_thm}. For the first claim, we shall adapt arguments from the literature on quadratic backward SDEs, as started by \citeauthor{kobylanski} \cite{kobylanski}. See also \citeauthor{zhang_BSDE} \cite[Lemma~7.3.2]{zhang_BSDE}.

\begin{proof}[Proof of Theorem \ref{existence_thm}{\rm{-(i)}}]

    For any $R > 0$, let $(Y^R, Z^R)$ be the unique solution of\hspace{.5pt} ${\rm{BSDE}}_{R}({\rm{T}}^{\delta}, f)$. Since we have $Y^{R'} \leq Y^{R}$ for any $R > R'$ by Lemma \ref{comp_principle}, the process $Y \coloneqq \lim_{R \rightarrow \infty} Y^R$ is well-defined. Fix $\eta > 0$ and the stopping time ${\rm{T}}^{\delta}_{\eta}$ defined in \eqref{tau_eta}. By Remark \ref{bounded_below_rmk}, Proposition \ref{lasry_lions_estimates} and the comparison result of Lemma \ref{comp_principle}, we have:
    \begin{equation}\label{estimates_R_eta}
    \begin{array}{cc}
         \displaystyle \frac{({\rm{ess}}\inf f(0, 0))^-}{\lambda_0}  \; \leq \;   Y^{R}_{t} \; \leq \; (C_1 + \varepsilon) \Phi_p(\delta_t) + \frac{c}{\lambda_0}  \, , & \text{for all } \, t \leq {\rm{T}}^{\delta}_{\eta} \, , \, R > 0 \, .
    \end{array}
    \end{equation}
    Let $T > 0$. By \eqref{estimates_R_eta} and since $\Phi_p(\delta_t) \leq \Phi_p(\eta)$ if $t \leq {\rm{T}}^{\delta}_{\eta}$, we can see that $\big\|Y^{R}_{\wedge {\rm{T}}^{\delta}_{\eta}} \big\|_{\infty}$ is uniformly bounded in $R > 0$. Therefore, by Lemma \ref{BMO_lemma}, we obtain:
    \begin{equation*}\label{uniform_bounds_existence}
         \sup_{R > 0} \, \big\|Z^{R} \big\|_{\mathbb{H}^{2}_{T \wedge {\rm{T}}^{\delta}_{\eta}}}  < \infty \, .
    \end{equation*}
    Since $\mathbb{H}^{2}_{T \wedge {\rm{T}}^{\delta}_{\eta}}$ is a Hilbert space, $Z^{R}$ converges weakly up to a subsequence to a process $Z^{(\eta)}$ in $\mathbb{H}^{2}_{T \wedge {\rm{T}}^{\delta}_{\eta}}$. We organize the remaining arguments into four steps. \\

    \vspace{-8pt}

    \noindent \underline{\textit{Step 1:}} We first show that the convergence of $Z^{R}$ to $Z^{(\eta)}$ is strong in $\mathbb{H}^{2}_{T \wedge {\rm{T}}^{\delta}_{\eta}}$. Let $\gamma > 0$ and introduce:
    \begin{equation}\label{phi_ODE}
        \begin{array}{ccccc}
           \displaystyle \varphi(y) = \big( e^{2 \gamma y} - 2 \gamma y - 1) / (2 \gamma)^2 \, , & \text{so that} & \displaystyle  \varphi'(y) = \frac{e^{2\gamma y} - 1}{2 \gamma} & \text{and} & \displaystyle \varphi''(y) = 2\gamma \varphi'(y) + 1 \, .
        \end{array}
    \end{equation}
    Let $R > R'$ and denote $\Delta Y^{R,R'} = Y^{R}- Y^{R'}$, $\Delta Z^{R, R'} = Z^{R}-Z^{R'}$, $\Delta Y^{R'} = Y^{R'} - Y$ and $\Delta Z^{R'} = Z^{R'} - Z^{(\eta)}$. Drawing from Assumption \ref{driver_assumptions} and \citeauthor{kobylanski} \cite[Proposition~2.4]{kobylanski}, there exists $C > 0$ such that, for all $y,y' \in \mathbb{R}$ and $z,z',z'' \in \mathbb{R}^d$:
    \begin{equation}\label{estimates_f}
        \big| f_t (y, z) - f_t(y', z')  \big| \leq C \big( 1 + |y| + |y'| + |z-z'|^2 + |z'-z''|^2 + |z''|^2 \big) \, .
    \end{equation}
    Now, applying Itô's formula to $\varphi(\Delta Y^{R, R'}_0)$ between $0$ and $T \wedge {\rm{T}}^{\delta}_{\eta}$, one has:
    \begin{align}
       \hspace{-5pt} \varphi(\Delta Y^{R, R'}_0) = \varphi(\Delta Y^{R, R'}_{T\wedge {\rm{T}}^{\delta}_{\eta}}) &- \int_{0}^{T\wedge {\rm{T}}^{\delta}_{\eta}} \! \varphi'(\Delta Y_s^{R, R'}) \Delta Z_s^{R, R'} \cdot dW_s - \frac{1}{2} \int_{0}^{T\wedge {\rm{T}}^{\delta}_{\eta}} \! \varphi''(\Delta Y_s^{R, R'}) \big|\Delta Z_s^{R, R'} \big|^2  \, ds\nonumber \\[3pt]
        &+ \int_{0}^{T\wedge {\rm{T}}^{\delta}_{\eta}} \! \varphi'(\Delta Y^{R, R'}_s) \big[f_s \big(Y_s^{R}, Z_s^{R} \big) -f_s \big(Y_s^{R'}, Z_s^{R'} \big) \big] \, ds \, . \nonumber 
    \end{align}
    So using \eqref{phi_ODE}, \eqref{estimates_f} and taking the expectation, we obtain:
    \begin{align}
        \varphi(\Delta Y^{R, R'}_{0}&) \; \leq \;  \mathbb{E} \big[  \varphi(\Delta Y^{R, R'}_{T\wedge {\rm{T}}^{\delta}_{\eta}}) \big] -\frac{1}{2} \mathbb{E} \int_{0}^{T\wedge {\rm{T}}^{\delta}_{\eta}} \! \big[ 2 \gamma \varphi'(\Delta Y_s^{R, R'}) + 1 \big] |\Delta Z_s^{R, R'}|^2 \, ds \nonumber \\[3pt]
        & + \mathbb{E} \int_{0}^{T\wedge {\rm{T}}^{\delta}_{\eta}} \!  C \varphi'(\Delta Y^{R, R'}_s) \Big[ 1 + \big|Y^{R}_s \big| + \big|Y^{R'}_s \big| + \big|\Delta Z^{R, R'}_s \big|^2 + \big|\Delta Z^{R'}_s \big|^2 + \big|Z_s^{(\eta)} \big|^2 \Big] \, ds  \, . \nonumber
    \end{align}
    Choosing $\gamma = 2C$ and using that $\varphi'(\Delta Y^{R, R'}_s) \geq 0$ because $Y^R \geq Y^{R'}$, we can write after rearranging:
    \begin{align}
       \mathbb{E} \int_{0}^{T\wedge {\rm{T}}^{\delta}_{\eta}}& \Big(\gamma \varphi'(\Delta Y_s^{R, R'}) + 1\Big) \big|\Delta Z_s^{R, R'} \big|^2 \, ds - \mathbb{E} \int_{0}^{T\wedge {\rm{T}}^{\delta}_{\eta}} \! \gamma \varphi'(\Delta Y_s^{R, R'}) \big|\Delta Z_s^{R'} \big|^2 \, ds \nonumber \\ 
       &\leq \,  2 \, \mathbb{E} \big[ \varphi(\Delta Y^{R, R'}_{T\wedge {\rm{T}}^{\delta}_{\eta}}) - \varphi(\Delta Y^{R, R'}_0) \big] + \mathbb{E} \int_{0}^{T\wedge {\rm{T}}^{\delta}_{\eta}} \! \gamma \varphi'(\Delta Y^{R, R'}_s) \Big[ 1 + \big| Y^{R}_s \big| + \big|Y^{R'}_s \big| +  \big|Z_s^{(\eta)} \big|^2 \Big] \, ds    \, . \nonumber 
    \end{align}
    Let $V^{R, R'} \coloneqq \gamma \varphi'(\Delta Y^{R, R'}) + 1$ and $V^{R'} \coloneqq \gamma \varphi'(\Delta Y^{R'}) + 1$. Since $\sqrt{V^{R,R'}} \Delta Z^{R,R'}$ converges weakly to $\sqrt{V^{R'}} \Delta Z^{R'}$ in $\mathbb{H}^{2}_{T \wedge {\rm{T}}^{\delta}_{\eta}}$ as $R \to \infty$, one has:
    \begin{equation*}
        \liminf_{R \to \infty} \| \sqrt{V^{R,R'}} \Delta Z^{R,R'} \|_{\mathbb{H}^{2}_{T \wedge {\rm{T}}^{\delta}_{\eta}}} \geq \| \sqrt{V^{R'}} \Delta Z^{R'} \|_{\mathbb{H}^{2}_{T \wedge {\rm{T}}^{\delta}_{\eta}}} \, .
    \end{equation*}
    Combining this with the Dominated Convergence Theorem for the other terms when $R \to \infty$,
    \begin{align}
        \big\| \Delta Z^{R'} \big\|^2_{\mathbb{H}^{2}_{T \wedge {\rm{T}}^{\delta}_{\eta}}}   \leq \,  2 \, \mathbb{E} \big[ \varphi(\Delta Y^{R'}_{T\wedge {\rm{T}}^{\delta}_{\eta}}) - \varphi(\Delta Y^{R'}_0) \big] + \mathbb{E} \int_{0}^{T\wedge {\rm{T}}^{\delta}_{\eta}}  \gamma \varphi'(\Delta Y^{R'}_s) \Big[ 1 + \big| Y_s \big| + \big|Y^{R'}_s \big| +  \big|Z_s^{(\eta)} \big|^2 \Big] \, ds    \, . \nonumber
    \end{align}
    The right-hand side converges to 0 as $R' \to \infty$ by Dominated Convergence and we obtain the result. \\

    \vspace{-8pt}

    \noindent \underline{\textit{Step 2:}} We now prove that $Y$ is a continuous process. Let $T, \eta > 0$, then one has:
    \begin{equation*}
       \hspace{-10pt} \sup_{t \leq T\wedge {\rm{T}}^{\delta}_{\eta}} \big| \Delta Y^R_t \big| \; \leq \;  \big| \Delta Y^R_{T \wedge {\rm{T}}^{\delta}_{\eta}} \big| + \int_0^{T\wedge {\rm{T}}^{\delta}_{\eta}} \big| f_s(Y_s^{R}, Z_s^{R}) - f_s(Y_s, Z_s^{(\eta)}) \big| ds + \sup_{t \leq T \wedge {\rm{T}}^{\delta}_{\eta}} \Big|\int_t^{T \wedge {\rm{T}}^{\delta}_{\eta}} \Delta Z^{R}_s \cdot dW_s \Big| \, .
    \end{equation*}
    By the Burkhölder-Davis-Gundy inequality, there exists a constant $C_{\eta} > 0$ depending only on $\eta$ and such that:
    \begin{equation*}
        \mathbb{E} \Bigg[\sup_{t \leq T \wedge {\rm{T}}^{\delta}_{\eta}} \Big|\int_t^{T \wedge {\rm{T}}^{\delta}_{\eta}} \Delta Z^{R}_s \cdot dW_s \Big| \Bigg] \leq C_{\eta} \,  \mathbb{E} \int_0^{T \wedge {\rm{T}}^{\delta}_{\eta}} \big|\Delta Z^{R}_s \big|^2 \, ds \, .
    \end{equation*}
    We conclude from Step 1 that $ \lim_{R \rightarrow \infty}\mathbb{E} \Big[ \,  \sup_{t \leq T \wedge {\rm{T}}^{\delta}_{\eta}}  \big| \Delta Y^R_t \big| \,  \Big] = 0$. In particular, $Y^{(\eta)}$ is an almost surely continuous process. \\

    \vspace{-8pt}
    
    \noindent \underline{\textit{Step 3:}} Notice that $Z_t^{(\eta)} = Z_t^{(\eta')}$ on $\{ {\rm{T}}^{\delta}_{\eta} \geq t \}$ if $\eta' > \eta$. As ${\rm{T}}^{\delta}_{\eta} \rightarrow {\rm{T}}^{\delta}$, we may define the process $Z_t \coloneqq Z^{(\eta)}_t$ on $\{ {\rm{T}}^{\delta}_{\eta} \geq t \}$ for all $\eta > 0$. Let us show that $(Y,Z)$ is a solution of the BSDE \eqref{BSDE_LL}. For all $0 \leq t \leq T$, one has:

    \vspace{-8pt}
    
    \begin{equation*}
        Y^{R}_{t \wedge {\rm{T}}^{\delta}_{\eta}} = Y^{R}_{T\wedge {\rm{T}}^{\delta}_{\eta}} - \int_{t\wedge {\rm{T}}^{\delta}_{\eta}}^{T\wedge {\rm{T}}^{\delta}_{\eta}} f_s \big( Y^{R}_s, Z^{R}_s \big) ds - \int_{t\wedge {\rm{T}}^{\delta}_{\eta}}^{T\wedge {\rm{T}}^{\delta}_{\eta}} Z^{R}_s \cdot dW_s \, . 
    \end{equation*}
    Letting $R \rightarrow \infty$ and applying the Dominated Convergence Theorem, we obtain:
    \begin{equation*}
        Y_{t \wedge {\rm{T}}^{\delta}_{\eta}} = Y_{T \wedge {\rm{T}}^{\delta}_{\eta}} - \int_{t\wedge {\rm{T}}^{\delta}_{\eta}}^{T\wedge {\rm{T}}^{\delta}_{\eta}} f_s(Y_s, Z_s)  ds - \int_{t\wedge {\rm{T}}^{\delta}_{\eta}}^{T\wedge {\rm{T}}^{\delta}_{\eta}} Z_s \cdot dW_s \, .
    \end{equation*}
    Finally, it follows from Lemma \ref{comp_principle} that, for all $\varepsilon > 0$ and $\eta \in [0, \overline{\eta})$, $Y \geq \underline{w}^{\varepsilon, \eta}$ for $t \in [0, {\rm{T}}^{\delta}]$. Combining with \eqref{estimates_R_eta}, we obtain the estimate \eqref{bounds_y}, which implies that $\lim_{\eta \rightarrow 0} Y_{{\rm{T}}^{\delta}_{\eta}} = +\infty$ a.s. \hspace{-11pt} on $\{ {\rm{T}}^{\delta} < \infty \}$, so $(Y,Z)$ solves\hspace{.5pt} ${\rm{BSDE}}_{\infty}({\rm{T}}^{\delta}, f)$ in the sense of Definition \ref{singular_def}.  \\

    \vspace{-8pt}

    \noindent \underline{\textit{Step 4:}} Let $(Y',Z')$ be another solution to \eqref{BSDE_LL} in the sense of Definition \ref{singular_def}, then the comparison result of Lemma \ref{comp_principle} gives $Y^R \leq Y'$ a.s. \hspace{-7pt} for all $R$. Then taking the limit $R \rightarrow \infty$ we obtain $Y \leq Y'$. Therefore $(Y, Z)$ is the minimal solution of\hspace{.5pt} ${\rm{BSDE}}_{\infty}({\rm{T}}^{\delta}, f)$.
\end{proof}

We now establish the second claim of Theorem \ref{existence_thm}. The proof is based on arguments from \citeauthor{popier_random_time} \cite[Proposition~4]{popier_random_time}.

\begin{proof}[Proof of Theorem \ref{existence_thm}{\rm{-(ii)}}]
Let $\kappa \in \mathcal{T}_0^{\theta}$ and $\eta, T > 0$. By \eqref{delta_BMO}, we can assume without loss of generality that $\delta \in \mathbb{S}^{\infty}_{\theta}$. Denoting $\tau_{\eta, T} \coloneqq \theta \wedge T \wedge {\rm{T}}^{\delta}_{\eta}$, we want to show that:
\begin{equation*}
        \mathsf{A}_{\eta, T}^{\kappa} \coloneqq \int_{\kappa}^{\tau_{\eta, T}} |Z_s|^2 \delta_s^{2 (r + \gamma)} ds \leq M \, ,
    \end{equation*}
for some constant $M > 0$ independent of $\kappa$, $\eta$ and $T$. In the rest of the proof, $M$ denotes a constant which may change from line to line. By Remark \ref{bounded_below_rmk}, note that $\overline{Y} \coloneqq Y - \frac{\overline{m}}{\lambda_0}  \geq 0$, where $\overline{m} = \frac{({\rm{ess}} \inf f(0,0) )^-}{\lambda_0}$.  Applying the Itô's formula to $\overline{Y}^2 \delta^{2(r + \gamma)}$ on $[\kappa, \tau_{\eta, T}]$, we obtain:
    \begin{align}
        \mathsf{A}_{\eta, T}^{\kappa} &+ \mathsf{B}_{\eta, T}^{\kappa} + \mathsf{C}_{\eta, T}^{\kappa} = \overline{Y}_{\tau_{\eta, T}}^2 \delta_{\tau_{\eta, T}}^{2 (r + \gamma)} - \overline{Y}_{\kappa}^2 \delta_{\kappa}^{2(r + \gamma)} - 2 \int_{\kappa}^{\tau_{\eta, T}} \overline{Y}_s \delta_s^{2(r + \gamma)} Z_s \cdot dW_s \nonumber \\[6pt]
        &- (2 r + \gamma) \int_{\kappa}^{\tau_{\eta, T}} \overline{Y}_s^2 \delta_s^{2 (r + \gamma)-1} \big[ \Xi_s ds + U_s \cdot dW_s \big] \nonumber - (r + \gamma) (2 (r + \gamma) -1) \int_{\kappa}^{\tau_{\eta, T}} \overline{Y}_s^2 \delta_s^{2 (r + \gamma -1)} | U_t |^2 ds \, , \nonumber
    \end{align}
    where we denoted:
    \begin{equation*}
        \begin{array}{lcr}
             \displaystyle \mathsf{B}_{\eta, T}^{\kappa} \coloneqq -2 \int_{\kappa}^{\tau_{\eta, T}} \overline{Y}_s \delta_s^{2 (r + \gamma)} f_s(Y_s , Z_s) ds \, , 
        & \text{and} & \displaystyle \mathsf{C}_{\eta, T}^{\kappa} \coloneqq  4(r + \gamma) \int_{\kappa}^{\tau_{\eta, T}} \overline{Y}_s \delta_s^{2(r + \gamma)-1} Z_s \cdot U_s ds \, .
        \end{array}
    \end{equation*}
    Since $\delta$ and $U$ are bounded on $[0, \theta]$ and $(\overline{Y}, Z) \in \mathbb{S}^{\infty}_{T \wedge {\rm{T}}^{\delta}_{\eta}} \times \mathbb{H}^2_{T \wedge {\rm{T}}^{\delta}_{\eta}}$, the Itô integrals are martingales. Moreover, $\Xi$ and $\overline{Y} \Phi_p(\delta)^{-1} $ belong to $ \mathbb{S}^{\infty}_{\theta}$ from the estimate \eqref{bounds_y}. Therefore, taking the expectation conditional to $\mathcal{F}_{\kappa}$ and rearranging the terms, we obtain:
    \begin{align}\label{finiteness_green}
        \mathbb{E}_{\kappa} \big[ \mathsf{A}_{\eta, T}^{\kappa} + \mathsf{B}_{\eta, T}^{\kappa} + \mathsf{C}_{\eta, T}^{\kappa} \big] \, \leq \,  M \Big( 1 + \mathbb{E}_{\kappa} \int_{\kappa}^{\theta} \big(\Phi_p(\delta) \delta_s^{r+ \gamma - 1} \big)^2 ds  \Big) \, \leq \, M \big( 1 +  \| \Phi_p(\delta) \delta_s^{r+ \gamma - 1} \|_{{\rm{BMO}}, \theta}^2 \big)  \, . \nonumber
    \end{align}
    Now we deal with the two terms $\mathsf{B}_{\eta, T}^{\kappa}$ and $\mathsf{C}_{\eta, T}^{\kappa}$. For $\mathsf{C}_{\eta, T}^{\kappa}$, the Cauchy-Schwarz inequality gives:
    \begin{align}
         \mathbb{E}_{\kappa} [\mathsf{C}_{\eta, T}^{\kappa}] \geq   - 4(r + \gamma) \Bigg( \mathbb{E}_{\kappa} \int_{\kappa}^{\tau_{\eta, T}} \! \overline{Y}_s^2 \delta_s^{2(r + \gamma-1)} |U_s|^2 ds \Bigg)^{\frac{1}{2}} \Bigg( \mathbb{E}_{\kappa} \int_{\kappa}^{\tau_{\eta, T}} \! |Z_s|^2 \delta_s^{2(r + \gamma)} ds \Bigg)^{\frac{1}{2}}  \geq  - M \, \mathbb{E}_{\kappa}[\mathsf{A}_{\eta, T}^{\kappa}]^{\frac{1}{2}} \, . \nonumber
    \end{align}
    For the term $\mathsf{B}_{\eta, T}^{\kappa}$, using the growth of $f$ in $y$ and $z$, \textit{i.e.} Assumption \eqref{p_growth}, and the fact that $\overline{Y} \geq 0$, we can write
    \begin{align}
        \mathbb{E}_{\kappa}[\mathsf{B}_{\eta, T}^{\kappa}] \, \geq \,  2 \mathbb{E}_{\kappa} \int_{\kappa}^{ \tau_{\eta, T}} \overline{Y}_s \delta_s^{2(r + \gamma)} \big\{ k_1 \big|Z_s\big|^p  - \lambda_1 |Y_s| - \vartheta \big\}  ds \, \geq \,  - M \Big( 1 +  \| \Phi_p(\delta) \delta_s^{r+ \gamma} \|_{{\rm{BMO}}, \theta}^2 \Big) \, ,  \nonumber
    \end{align}
    Finally, we obtain a system of inequalities of the form:
    \begin{equation*}
    \begin{array}{ccc}
         \mathbb{E}_{\kappa}[\mathsf{A}_{\eta, T}^{\kappa}] \geq 0 & \text{and} & \mathbb{E}_{\kappa}[\mathsf{A}_{\eta, T}^{\kappa}] - M - M \, \mathbb{E}_{\kappa}[\mathsf{A}_{\eta, T}^{\kappa}]^{\frac{1}{2}} \leq \mathbb{E}_{\kappa} \big[ \mathsf{A}_{\eta, T}^{\kappa} + \mathsf{B}_{\eta, T}^{\kappa} + \mathsf{C}_{\eta, T}^{\kappa} \big] \leq M \, .
    \end{array}
    \end{equation*}
    By Young's inequality, we deduce that $\mathbb{E}_{\kappa}[\mathsf{A}_{\eta, T}^{\kappa}] \leq M$, once we adjust the constant $M$. Note that $M$ is independent of $\kappa$, $\eta$ and $T$. Applying Fatou's Lemma,
    \begin{equation*}
       \mathbb{E}_{\kappa} \int_{\kappa}^{\theta} |Z_s|^2 \delta_s^{2(r + \gamma)} ds \leq \liminf_{\eta \rightarrow 0} \liminf_{T \to \infty} \,  \mathbb{E}_{\kappa}[\mathsf{A}_{\eta, T}^{\kappa}] \leq M \, ,
    \end{equation*}
    which proves that $Z \delta^{r+\gamma} \in \mathbb{H}^{\rm{BMO}}_{\theta}$.
\end{proof}

\begin{rmk}
    The \rm{BMO} estimate on $Z$ is natural. Indeed, it is satisfied for sub- and supersolutions of the form $(\overline{w}_t, \overline{z}_t) = \big( C \Phi_p(\delta) \pm c \, , \, C \Phi'_p(\delta) \big)$, since $\| \overline{z} \delta^{r + \gamma} \|_{{\rm{BMO}}, \theta} \leq  C \| \delta^{\gamma - 1} \|_{{\rm{BMO}}, \theta} < \infty$.
\end{rmk}

\subsection{Uniqueness of solutions}\label{uniqueness_control_sec}

This section is devoted to the proof of Theorem \ref{uniqueness_thm} on the uniqueness of solutions of \eqref{BSDE_LL}. We suppose that Assumptions \ref{d_assumptions} and \ref{driver_assumptions} hold.

\vspace{5pt}

\begin{proof}[Proof of Theorem \ref{uniqueness_thm}]
    We proceed in two steps by adapting the arguments of \citeauthor{Lasry_Lions_state_constraints} \cite[Theorem~II.1]{Lasry_Lions_state_constraints} and \cite[Theorem~III.4]{Lasry_Lions_state_constraints} to our framework. Denote by $Y^{\text{min}}$ the minimal solution constructed in Theorem \ref{existence_thm}. First we identify a maximal solution $Y^{\text{max}}$ which has the same asymptotic behavior \eqref{asymptotics_y} as $Y^{\rm{min}}$. Then we use the concavity of $f$ and our comparison result to conclude that $Y^{\text{max}} \leq Y^{\text{min}}$, which implies the uniqueness of solutions. \\

    \vspace{-8pt}

    \noindent \underline{\textit{Step 1:}} Fix $\eta \in (0, \delta_0)$ and let ${\rm{T}}^{\delta}_{\eta} = \inf \big \{ t \geq 0 : \delta^{\eta}_t \leq 0 \big \}$ with $\delta^{\eta}_t \coloneqq \delta_t - \eta$. As $\delta^{\eta}$ also satisfies Assumption \ref{d_assumptions}, we deduce from Theorem \ref{existence_thm} the existence of a minimal solution $\big( \overline{Y}^{\eta}, \overline{Z}^{\eta} \big)$ of ${\rm{BSDE}}_{\infty}({\rm{T}}^{\delta}_{\eta}, f)$ satisfying $\underline{w}^{\varepsilon, -\eta} \leq \overline{Y}^{\eta} \leq \overline{w}^{\varepsilon, \eta}$ a.s. By the comparison result of Lemma \ref{comp_principle}, $Y' \leq \overline{Y}^{\eta}$ and the sequence $(\overline{Y}^{\eta})$ is non-decreasing, \textit{i.e.} $\overline{Y}^{\eta} \leq \overline{Y}^{\eta'}$ when $\eta < \eta'$. Then $Y' \leq Y^{\text{max}} \coloneqq \lim_{\eta \rightarrow 0} \overline{Y}^{\eta}$. With an argument similar to the proof of Theorem \ref{existence_thm}, we obtain a process $Z^{\text{max}}$ such that $(Y^{\text{max}}, Z^{\text{max}})$ is the maximal solution to the BSDE \eqref{BSDE_LL}. To sum up, there exists $c_{\varepsilon} > 0$ such that:
    \begin{equation}\label{bounds_y_max}
        \begin{array}{cc}
          \displaystyle (C_0 - \varepsilon) \Phi_p(\delta_t) - c_{\varepsilon} \leq Y^{\text{min}} \leq Y' \leq Y^{\text{max}} \leq (C_1 - \varepsilon) \Phi_p(\delta_t) + c_{\varepsilon} \,, & \text{a.s.}
    \end{array}
    \end{equation}

    \noindent \underline{\textit{Step 2:}} We now show that $Y^{\text{max}} \leq Y^{\text{min}}$. Let $\Tilde{C_0}, \Tilde{C_1}$ be constants satisfying $\Tilde{C_0} < C_0 \leq C_1 < \Tilde{C_1}$. Then by \eqref{bounds_y_max}, there exist $c> 0$ and $\theta \in (0, 1)$ such that:
    \begin{equation*}
       \theta (Y^{\text{max}} + 1+ c ) \, \leq \, \theta ( \Tilde{C_1} \Phi_p(\delta_t) + 1+ 2c ) \, \leq \,  \Tilde{C_0} \Phi_p(\delta_t) + 1 \, \leq \,  Y^{\text{min}} + 1 + c \,.
    \end{equation*}
    Let $\eta_0 > 0$. By Proposition \ref{lasry_lions_estimates}, we may choose $c$ large enough so that the processes
    \begin{equation*}
    \begin{array}{cc}
         \displaystyle (y^{\eta}_t \, , \, z^{\eta}_t) \coloneqq \big(\Tilde{C_0} \Phi_p(\delta_t + \eta) - c  \, , \, \Tilde{C_0} \Phi_p'(\delta_t + \eta) \big) \, , & \text{ for } \, \eta \in (0, \eta_0) \,,
    \end{array}
    \end{equation*}
    are subsolutions of ${\rm{BSDE}}_{\infty}({\rm{T}}^{\delta}, f)$. Now define $\theta_0 \in (0, 1)$ by:
    \begin{equation*}
        \theta_0 \coloneqq \sup \Big \{ \theta \in (0, 1] : \theta Y^{\text{max}} \leq Y^{\text{min}} + (1-\theta)(1+c) \Big\} \, .
    \end{equation*}
    We will show that $\theta_0 = 1$, which concludes the proof. By contradiction, we assume that $\theta_0 < 1$. Then, for $\theta \in (0, \theta_0)$, we define:
    \begin{equation}\label{y_theta_def}
    \begin{array}{ccc}
         \displaystyle Y^{\theta, \eta} \coloneqq \theta \, Y^{\text{max}} + (1- \theta) y^{\eta} & \text{and} & \displaystyle Z^{\theta, \eta} \coloneqq \theta Z^{\text{max}} + (1-\theta) z^{\eta} \, .
    \end{array}
    \end{equation}
    Since $f$ is concave in $z$, we conclude that $(Y^{\theta, \eta}, Z^{\theta, \eta})$ is a subsolution of ${\rm{BSDE}}_{\infty}({\rm{T}}^{\delta}, f)$. Indeed, denoting
    \begin{align}
        K^{\theta, \eta}_t \coloneqq Y^{\theta, \eta}_0 - Y^{\theta, \eta}_t - \int_0^t f_s(Y^{\theta, \eta}_s, Z^{\theta, \eta}_s) \, ds + \int_0^t Z^{\theta, \eta}_s \cdot dW_s \,, \nonumber
    \end{align}
    we obtain $\dot{K}^{\theta, \eta}_t \leq \theta f_t(Y_t^{\text{max}}, Z_t^{\text{max}}) + (1-\theta) f_t(y^{\eta}, z^{\eta}) - f_t(Y^{\theta, \eta}, Z^{\theta, \eta}) \leq 0$. Notice moreover that:
    \begin{equation*}
    \begin{array}{cc}
         \displaystyle Y^{\theta, \eta}_{{\rm{T}}^{\delta}_{\eta'}} \leq Y^{\text{min}}_{{\rm{T}}^{\delta}_{\eta'}} - (\theta_0 - \theta) Y^{\text{max}}_{{\rm{T}}^{\delta}_{\eta'}} + (1-\theta) \big( y^{\eta}_{{\rm{T}}^{\delta}_{\eta'}} + 1 + c \big) \,, & \text{ for all } \eta' > 0 \, .
    \end{array}
    \end{equation*}
    Since $\theta < \theta_0$, the right-hand side is nonpositive for $\eta'$ small enough by \eqref{bounds_y_max}. We deduce by Lemma \ref{comp_principle} that $Y^{\theta, \eta}_{t} \leq Y_{t}^{\text{min}}$ for all $t \leq {\rm{T}}^{\delta}_{\eta'}$, and then $Y^{\theta, \eta} \leq Y^{\text{min}}$ for all $\eta \in (0, \eta_0)$ and $\theta \in (0, \theta_0)$. By the same argument as above, there exists $\nu \in (0, 1)$ such that $\nu (Y^{\text{max}} + 1 + c) \leq \Tilde{C_0} \Phi_p(\delta) + 1$. Then, sending $\theta \to \theta_0$ and $\eta \to 0$, this yields:
    \begin{equation*}
       \big(\theta_0 + (1-\theta_0) \nu \big) Y^{\text{max}}_t \,\leq \,  \theta_0 \, Y^{\text{max}}_t + (1- \theta_0) (\Tilde{C_0} \Phi_p(\delta_t) +1) \,\leq \, Y^{\text{min}}_t + \big(1-\theta_0-(1-\theta_0)\nu \big) (1+c) \,.
    \end{equation*}
    This contradicts the definition of $\theta_0$, which ends the proof.
\end{proof}

\section{Connection with optimal control}\label{Control_sec}

In this section, we prove a verification argument which describes the minimal solution $(Y, Z)$ of the BSDE \eqref{BSDE_LL} as the value function of a stochastic optimal control problem. Recall that the set of admissible processes $\mathcal{A}$ was introduced in Definition \ref{A_def}, and recall that $p' = \frac{p}{p-1}$. The next result presents some basics properties of the cost function $g$.

\begin{lemma}\label{g_basics_lemma}
    Under Assumption {\rm{\ref{control_assumption}}}, $g$ is bounded from below, and $a \mapsto g_t(a)$ is of class $C^1$ and satisfies, for all $a, a' \in \mathbb{R}^d$,
    \begin{equation*}
        \begin{array}{cc}
             \displaystyle g_t(a')  \ge g_t(a) + \nabla_{\!a} g_t(a) \cdot (a'-a) + \frac{p-1}{p^{p'} \nu^{\frac{1}{p-1}}} |a'-a|^{p'} \, , & d\mathbb{P} \otimes dt- \text{a.e.}
        \end{array}
    \end{equation*}
    Moreover, $\hat{\alpha}_t \coloneqq \nabla_{\! z} H_t(Z_t)$ is the unique $\mathbb{F}$-measurable process satisfying $g_t(\hat{\alpha}_t) = \hat{\alpha}_t \cdot Z_t - H_t(Z_t)$.
\end{lemma}

\begin{proof}
    For the first claim, notice that:
    \begin{equation}\label{g_bounded_below}
    \begin{array}{cc}
         \displaystyle g_t(a) \, \geq \,  \sup_{z \in \mathbb{R}^d} \{ a \cdot z - \kappa_0 |z|^p - \vartheta  \} \, \geq \, \frac{p^{p'}}{p'}\kappa_0^{1-p'} |a|^{p'} - \vartheta  \, , & \text{for all } \, a \in \mathbb{R}^d \, .
    \end{array}
    \end{equation}
    The other two claims are standard consequences of the convexity and $p$-smoothness of $H$.
\end{proof}

The fact that $\hat{\alpha}$ defined in Lemma \ref{g_basics_lemma} is an admissible process is not completely straightforward. First of all, note that $\theta_{\eta} \coloneqq {\rm{T}}^{\delta}_{\eta} \wedge \eta^{-1}$ is a bounded stopping time such that $Y_{\theta_{\eta}}$ is bounded and $\theta_{\eta} \rightarrow {\rm{T}}^{\delta}$ a.s. \hspace{-7pt} as $\eta \rightarrow 0$.

\begin{lemma}\label{alpha_eta}
    For all $\eta > 0$, $\hat{\alpha} \in \mathbb{H}^{\text{BMO}}_{\theta_{\eta}}$ and there exists a unique probability distribution $\mathbb{P}^{\hat{\alpha}, \eta}$ equivalent to $\mathbb{P}$ such that $W^{\hat{\alpha}, \eta} = W_t +  \int_0^t \hat{\alpha}_s \mathbbm{1}_{\{\theta_{\eta} > s\}} \, ds$ is a $\mathbb{P}^{\hat{\alpha}, \eta}$-Brownian motion.
\end{lemma}
\begin{proof}
    Note that $Z_t = \nabla_{\! a} g_t(\hat{\alpha}_t)$, so by Lemma \ref{g_basics_lemma}, we have $d\mathbb{P} \otimes dt$-a.e., 
    \begin{align}
        g_t(0) &\geq  g_t(\hat{\alpha}_t) - \hat{\alpha}_t \cdot Z_t + \frac{p-1}{p^{p'}\nu^{p'-1}} |\hat{\alpha}_t|^{p'} \nonumber \\[3pt]
        & = - H_t(Z_t) + \frac{p-1}{p^{p'}\nu^{p'-1}} |\hat{\alpha}_t|^{p'} \geq - k_0 |Z_t|^p - \vartheta  + \frac{p-1}{p^{p'}\nu^{p'-1}}  |\hat{\alpha}_t|^{p'} \nonumber \, .
    \end{align}
    This reformulates as $|\hat{\alpha}_t|^{p'} \leq \frac{p^{p'}\nu^{p'-1}}{p-1} (g_t(0) + k_0 |Z_t|^p)$. By Lemma \ref{BMO_lemma} and the fact that $\theta_{\eta}$ is a bounded stopping time, $Z \in \mathbb{H}^{\text{BMO}}_{\theta_{\eta}}$. Then, since $p' \geq 2$ and $g_t(0)$ is bounded, we deduce that $\hat{\alpha} \in \mathbb{H}^{\text{BMO}}_{\theta_{\eta}}$ and the statement follows from Girsanov's Theorem.
\end{proof}

In order to prove that $\hat{\alpha} \in \mathcal{A}$, we use a result borrowed from \citeauthor{strook_varadhan} \cite[Theorem~1.3.5]{strook_varadhan}, stated in the canonical space $(\Omega, \mathcal{F}) = (\mathscr{C}_d, \mathcal{B}(\mathscr{C}_d))$. We recall it here without proof, noting that we can always come back to this case, since $(\mathcal{F}_t)_{t \ge 0}$ is the Brownian filtration.

\begin{lemma}[\cite{strook_varadhan}]\label{strook_varadhan_lemma}
    Let $(\tau_n )$ be a non-decreasing sequence of stopping times adapted to the filtration $\mathbb{F}$, and $P_n$ be a probability distribution on $(\Omega, \mathcal{F}_{\tau_n})$. Assume that $P_{n+1} \equiv P_n$ on $\mathcal{F}_{\tau_n}$ for any $n$ and $\displaystyle P_n[\tau_n \leq T] \to 0$ as $n \to \infty$, for all $T > 0$. Then there exists a unique probability measure $P$ on $(\Omega, \mathcal{F})$ such that $P \equiv P_n$ on each $\mathcal{F}_{\tau_n}$. 
\end{lemma}

Before proving Theorem \ref{verif_arg}, we make the useful observation that the optimal control problem $\inf_{\alpha} J(\alpha)$ is actually a convex problem.

\begin{proposition}\label{P_A_convex}
    The set $\mathcal{P}_{\! \mathcal{A}} \coloneqq \{ \mathbb{P}^{\alpha} , \alpha \in \mathcal{A}\}$ is convex. Moreover, if $\alpha^1, \alpha^2 \in \mathcal{A}$ and $\mathbb{P}^{\alpha^1} = \mathbb{P}^{\alpha^2}$, then $\alpha^1 = \alpha^2$, $d\mathbb{P} \otimes dt-$a.e. on $[0, {\rm{T}}^{\delta})$.
\end{proposition}

\begin{proof}
    Let $\alpha^1, \alpha^2 \in \mathcal{A}$ and denote by $L^{i}$ the changes of measures:
    \begin{equation*}
        \begin{array}{cc}
             \displaystyle L^{i}_t \coloneqq \mathcal{E}\Big( -\int_0^{\cdot} \alpha_s \cdot dW_s \Big)_{t} = \frac{d \mathbb{P}^{\alpha^{i}}}{d \mathbb{P}} \Big|_{\mathcal{F}_t} \, , & i \in \{ 1, 2\} \,, \, t < {\rm{T}}^{\delta} .
        \end{array}
    \end{equation*}
    Then $L^{i}$ satisfies $d L^{i}_t = - L^{i}_t \alpha^{i}_t \cdot dW_t$ when $t < {\rm{T}}^{\delta}$.
    Let $\varepsilon \in (0,1)$ and define:
    \begin{equation*}
        \begin{array}{ccc}
             \displaystyle L^{\varepsilon} \coloneqq \varepsilon L^1 + (1-\varepsilon) L^2 \, , & \displaystyle\alpha^{\varepsilon} \coloneqq \frac{\varepsilon L^1 \alpha^1 + (1-\varepsilon) L^2 \alpha^2}{\varepsilon L^1 + (1-\varepsilon) L^2} & \text{and } \, \mathbb{P}^{\varepsilon} \coloneqq \varepsilon \mathbb{P}^{\alpha^1} + (1-\varepsilon) \mathbb{P}^{\alpha^2} \,.
        \end{array}
    \end{equation*}
    Notice that $\alpha^{\varepsilon} = \pi_t \alpha^1_t + (1-\pi_t) \alpha^2_t$, with $\pi_t \coloneqq \frac{\varepsilon L^1_t}{L^{\varepsilon}_t} \in [0, 1]$, so that $\int_0^{T \wedge {\rm{T}}^{\delta}_{\eta}} |\alpha^{\varepsilon}_t|^2 dt < \infty$, $\mathbb{P}-$a.s. for all $T, \eta > 0$. Moreover, $L^{\varepsilon}$ is a martingale and $d L_t^{\varepsilon} = - L_t^{\varepsilon} \alpha_t^{\varepsilon} \cdot dW_t$, so we have:
    \begin{equation*}
       \frac{d \mathbb{P}^{\alpha^{\varepsilon}}}{d \mathbb{P}} \big|_{\mathcal{F}_{T \wedge {\rm{T}}^{\delta}_{\eta}}} = L^{\varepsilon}_{T \wedge {\rm{T}}^{\delta}_{\eta}} = \frac{d \mathbb{P}^{\varepsilon}}{d \mathbb{P}} \big|_{\mathcal{F}_{T \wedge {\rm{T}}^{\delta}_{\eta}}}  \, \text{ for all } \, T, \eta > 0 \, , \text{ which yields } \, \mathbb{P}^{\varepsilon} = \mathbb{P}^{\alpha^{\varepsilon}} \, .
    \end{equation*}
    We finally observe that $\mathbb{P}^{\varepsilon}[{\rm{T}}^{\delta} = \infty] = 1$, so $\mathbb{P}^{\varepsilon} \in \mathcal{P}_{\! \mathcal{A}}$, which proves the first claim. To show the second claim, we note that, if $\mathbb{P}^{\alpha^1} = \mathbb{P}^{\alpha^2}$, then $\alpha^1 = \alpha^2$, $d\mathbb{P} \otimes dt-$a.e. on $[0, T \wedge {\rm{T}}^{\delta}_{\eta}]$. We conclude by sending $T \to \infty$ and $\eta \to 0$.
\end{proof}

\vspace{5pt}

\begin{proof}[Proof of Theorem \ref{verif_arg}]
   We proceed in three steps. The first two steps establish the verification, while the last step shows the uniqueness of the optimal control. \\

    \vspace{-8pt}
    
    \noindent \underline{\textit{Step 1:}} Fixing $\eta > 0$ and the distribution $\mathbb{P}^{\hat{\alpha}, \eta}$ in Lemma \ref{alpha_eta}, we obtain:
    \begin{align}
        Y_0 &= e^{- \lambda \theta_{\eta}} Y_{\theta_{\eta}} + \int_0^{\theta_{\eta}} e^{-\lambda s} \big(  -H_s( Z_s) + \hat{\alpha}_s \cdot Z_s \big) ds - \int_0^{\theta_{\eta}} e^{-\lambda s} Z_s \cdot dW_s^{\hat{\alpha}, \eta} \nonumber \\
        &= e^{- \lambda \theta_{\eta}} Y_{\theta_{\eta}} + \int_0^{\theta_{\eta}} e^{-\lambda s} g_s(\hat{\alpha}_s) ds -  \int_0^{\theta_{\eta}} e^{-\lambda s} Z_s \cdot dW_s^{\hat{\alpha}, \eta} \, . \nonumber
    \end{align}
    Then, taking the expectation with respect to $\mathbb{P}^{\hat{\alpha}, \eta}$ and noting that $g$ is bounded from below by \eqref{g_bounded_below}:
    \begin{align}
        Y_0  \, = \, \mathbb{E}^{\mathbb{P}^{\hat{\alpha}, \eta}} \Big[ e^{- \lambda \theta_{\eta}} Y_{\theta_{\eta}} + \int_0^{\theta_{\eta}} e^{-\lambda s} g_s(\hat{\alpha}_s) ds \Big] \, \geq \,  \mathbb{E}^{\mathbb{P}^{\hat{\alpha}, \eta}} \big[ e^{- \lambda \theta_{\eta}} Y_{\theta_{\eta}} \big] + \frac{({\rm{ess}}\inf g )^-}{\lambda} \, . \label{Y_bounded_above_control}
    \end{align} 
    For $n>0$, define $P_n \coloneqq \mathbb{P}^{\hat{\alpha}, \frac{1}{n}}$ and $\tau_n \coloneqq \theta_{1/n}$. Notice that $P_{n+1} = P_n$ on $\mathcal{F}_{\tau_n}$. To apply Lemma \ref{strook_varadhan_lemma}, let $T > 0$ and choose $n > T$ so that $\tau_n = {\rm{T}}^{\delta}_{1/n}$, $P_n-$a.s. \hspace{-7pt} on $\{ \tau_n \leq T \}$. By the estimates \eqref{bounds_y}, \eqref{Y_bounded_above_control}, and the fact that $Y \geq \frac{({\rm{ess}}\inf f(0,0))^-}{\lambda}$ by Remark \ref{bounded_below_rmk}, we obtain for all $\varepsilon > 0$,
    \begin{align}
        Y_0 -  \frac{({\rm{ess}}\inf g )^-}{\lambda} \geq  \mathbb{E}^{P_n} \big[ e^{- \lambda \tau_n} \, Y_{\tau_n} \big] &=   \mathbb{E}^{P_n} \big[ e^{- \lambda {\rm{T}}^{\delta}_{1/n}} \,  Y_{{\rm{T}}^{\delta}_{1/n}} \mathbbm{1}_{\{\tau_n \leq T\}} \big] + \mathbb{E}^{P_n} \big[ e^{- \lambda \tau_n} \,  Y_{\tau_n} \mathbbm{1}_{\{\tau_n > T \}}  \big]   \nonumber \\
        & \geq (C_0 - \varepsilon)\, \Phi_p\big( \frac{1}{n} \big) \, e^{-\lambda T} \, \mathbb{P}_n [ \tau_n \leq T \big] - \frac{c_{\varepsilon}}{\lambda_1} + \frac{({\rm{ess}}\inf f(0,0))^-}{\lambda} \,. \nonumber
    \end{align}
    for some constant $c_{\varepsilon} > 0$ independent of $n$. This implies that $\mathbb{P}_n [ \tau_n \leq T \big] \rightarrow 0$ as $n \rightarrow \infty$. By Lemma \ref{strook_varadhan_lemma}, we obtain a unique probability measure $\mathbb{P}^{\hat{\alpha}}$ on $(\Omega, \mathcal{F})$ such that $\mathbb{P}^{\hat{\alpha}} \equiv P_n$ on each $\mathcal{F}_{\tau_n}$. It follows that $\mathbb{P}^{\hat{\alpha}}[ \tau_n \leq T] \to 0$ as $n \to \infty$, which imply that:
    \begin{equation*}
        \mathbb{P}^{\hat{\alpha}}[{\rm{T}}^{\delta} \leq  T] \, \leq \, \mathbb{P}^{\hat{\alpha}}[\tau_n \leq T]  \xrightarrow[n \to \infty]{} 0 \, .
    \end{equation*}
    Since this holds for any arbitrary $T > 0$, we conclude that $\mathbb{P}^{\hat{\alpha}}[{\rm{T}}^{\delta} = \infty] = 1$, so $\hat{\alpha} \in \mathcal{A}$. Finally we have, by Remark \ref{bounded_below_rmk},
    \begin{align}
        Y_0 \, = \, \mathbb{E}^{\mathbb{P}^{\hat{\alpha}}} \Big[ e^{- \lambda \tau_n} Y_{\tau_n} + \int_0^{\tau_n} e^{-\lambda s} g_s(\hat{\alpha}_s) ds \Big] \, \geq \,   \mathbb{E}^{\mathbb{P}^{\hat{\alpha}}} \Big[ e^{- \lambda \tau_n} \frac{({\rm{ess}}\inf f(0, 0))^-}{\lambda} + \int_0^{\tau_n} e^{-\lambda s} g_s(\hat{\alpha}_s) ds \Big]  \nonumber \, .
    \end{align}
    Since $e^{-\lambda \tau_n} \to 0$ and $\int_0^{\tau_n} e^{-\lambda s} g_s(\hat{\alpha}_s) ds \to \int_0^{\infty} e^{-\lambda s} g_s(\hat{\alpha}_s) ds$, $\mathbb{P}^{\hat{\alpha}}-$a.s. \hspace{-12pt} when $n \to \infty$, we decude by Fatou's Lemma that $Y_0 \geq J(\hat{\alpha})$. \\

    \vspace{-8pt}
    
    \noindent \underline{\textit{Step 2:}} Let $\alpha \in \mathcal{A}$, it is enough to prove that $Y_0 \leq J(\alpha)$ to conclude. Because $(Y, Z)$ was chosen as the minimal solution of the singular BSDE \eqref{BSDE_LL}, we know that $\displaystyle Y = \lim_{R \rightarrow \infty} Y^R$ where $(Y^R, Z^R)$ is the unique solution of\hspace{.5pt} ${\rm{BSDE}}_{R}({\rm{T}}^{\delta}, f)$. Then, if $(\theta_{\eta})_{\eta > 0}$ is a sequence of bounded stopping times such that $\theta_{\eta} \rightarrow {\rm{T}}^{\delta}$ as $\eta \to 0$, we may write:
    \begin{equation*}
        Y_0^R = e^{- \lambda \theta_{\eta}} Y_{\theta_{\eta}}^R + \int_0^{\theta_{\eta}} e^{-\lambda s} \big\{ - H_s(Z_s^R) + \alpha_s \cdot Z_s^R \big\} ds - \int_0^{\theta_{\eta}} e^{-\lambda s} Z_s^R \cdot dW_s^{\alpha} \, ,
    \end{equation*}
    where $W^{\alpha} \coloneqq W - \int_0^{\cdot} \alpha_s ds$ is a $\mathbb{P}^{\alpha}$-Brownian motion. Note that $e^{- \lambda \theta_{\eta}} Y_{\theta_{\eta}}^R \to 0$, $\mathbb{P}^{\alpha}-$a.s. \hspace{-10pt} since $Y^R$ is bounded and $\mathbb{P}^{\alpha}[{\rm{T}}^{\delta} = \infty] = 1$. By definition of the Hamiltonian $H$ and taking the expectation with respect to $\mathbb{P}^{\alpha}$, this leads to:
    \begin{equation*}
        Y_0^R \, \leq \,   \mathbb{E}^{\mathbb{P}^{\alpha}} \left[ e^{-\lambda \theta_{\eta}} Y_{\theta_{\eta}}^R + \int_0^{\theta_{\eta}} e^{-\lambda s} g_s(\alpha_s) ds \right] \xrightarrow[\eta \rightarrow 0]{} J(\alpha) \, ,
    \end{equation*}
    where the convergence holds by the Dominated and Monotone Convergence Theorems, together with the fact that $g$ is bounded from below. The claim is proved by sending $R \to \infty$. \\

    \vspace{-8pt}

    \noindent \underline{Step 3:} Keeping the notations of statement and proof of Proposition \ref{P_A_convex}, we show that the function $\Tilde{J}(\mathbb{P})$ defined on $\mathcal{P}_{\! \mathcal{A}}$ by $\Tilde{J}(\mathbb{P}^{\alpha}) \coloneqq J(\alpha)$, which is well defined by Proposition \ref{P_A_convex}, is strictly convex. Let $\alpha^1, \alpha^2 \in \mathcal{A}$ be two admissible processes. We use the notations from the proof of Proposition \ref{P_A_convex}. By convexity of $a \mapsto g_t(a)$, if $\varepsilon \in (0,1)$, we have:
    \begin{align}
        L^{\varepsilon}_t g_t(\alpha_t^{\varepsilon}) \, \leq \,  L^{\varepsilon}_t \pi_t g_t(\alpha^1_t) + L^{\varepsilon}_t (1-\pi_t) g_t(\alpha^2_t) \, \leq \,  \varepsilon L^1_t g_t(\alpha^1_t) + (1-\varepsilon) L^2_t g_t(\alpha^2_t) \, . \nonumber
    \end{align}
    Integrating, we obtain:
    \begin{equation*}
        \Tilde{J}(\mathbb{P}^{\varepsilon}) = \mathbb{E}^{\mathbb{P}} \! \int_0^{\infty} \! e^{-\lambda t} L^{\varepsilon}_t g_t(\alpha_t) dt \leq \varepsilon  \Tilde{J}(\mathbb{P}^{\alpha^1}) + (1-\varepsilon) \Tilde{J}(\mathbb{P}^{\alpha^2}) \, .
    \end{equation*}
    Moreover, since $a \mapsto g_t(a)$ is strictly convex by Lemma \ref{g_basics_lemma} and $L^1_t, L^2_t > 0$ a.s. on every finite horizon, the inequality above is strict if $\mathbb{P}^{\alpha^1} \neq \mathbb{P}^{\alpha^2}$. Therefore, if $\alpha^1$ and $\alpha^2$ satisfy $J(\alpha^1) = J(\alpha^2) = Y_0$, then $\alpha^1 = \alpha^2$, $d\mathbb{P} \otimes dt-$a.e. on $[0, {\rm{T}}^{\delta})$ by Proposition \ref{P_A_convex}.
\end{proof}

\section{Markov exit time}\label{Applications_sec}

In this section we prove the results of Section \ref{markov_main_results} under Assumption \ref{state_process_assumptions}. For $\eta \geq 0$, recall that $D^{\eta} = \{ x' \in D : |d_{\pm}(x')| < \eta \}$. Then for $(t, x) \in [0, \infty) \times D$, we denote by $X^{t, x}$ the unique solution of the SDE \eqref{SDE_state} on $[t, +\infty)$ such that $X^{t, x}_t = x$ and $\tau^{t, x} \coloneqq \inf \{ s \geq t : X_s^{t, x} \notin D \}$. When it is clear and convenient, we omit the dependence on $\eta$, $t$ and $x$ in the notations. We first state an extension of a result derived by \citeauthor{popier_random_time} in the proof of \cite[Proposition~4]{popier_random_time}.

\begin{lemma}\label{lemma_popier_dist}
    Let $D$ be a bounded domain of $\mathbb{R}^d$ with $C^2$-regular boundary.
    Let $x \in D$ and $\mathbf{d}$ be the regularized signed distance to $\partial D$ defined in \eqref{regularized_distance}. Then, for all $\gamma > \frac{1}{2}$, $\mathbf{d}(X_{\cdot}^{0, x})^{\gamma - 1} \in \mathbb{H}_{\tau^{0, x}}^{\textnormal{BMO}}$.
\end{lemma}

\begin{proof}
    Let $\kappa \in \mathcal{T}_0^{\tau^{0, x}}$. By the strong Markov property of $X$ and the fact $\tau^{0, x} = \kappa + \tau^{\kappa, X^{0, x}_{\kappa}}$, we have:
    \begin{align}
       \Big\| \mathbb{E}_{\kappa} \int_{\kappa}^{\tau^{0, x}} \!\! \mathbf{d}(X_s^{0, x})^{2(\gamma-1)} \, ds \Big\|_{\infty} \hspace{-5pt} &=  \sup_{x' \in D} \mathbb{E} \Big[\int_{\kappa}^{\kappa + \tau^{\kappa, x'}} \! \!\!\!\! \mathbf{d}(X_s^{\kappa, x'})^{2(\gamma-1)} \, ds \, \Big| \,  X_{\kappa}^{0, x} = x' \Big] \nonumber \\[3pt]
       &=  \sup_{x' \in D} \mathbb{E} \Big[\int_{0}^{{\tau^{0, x'}}} \!\! \mathbf{d}(X_s^{0, x'})^{2(\gamma-1)} \, ds \, \Big| \, X_{0}^{0, x'} = x' \Big] \, . \nonumber
    \end{align}

    Let $G(x, \cdot)$ be the Green function of the process $X^{0, x}$ killed at time $\tau^{0, x}$, see \citeauthor{pinsky_book} \cite[Section 4.2, Theorem 2.5]{pinsky_book}. Then there exists $C > 0$ such that:
    \begin{equation*}
        \begin{array}{lc}
             \displaystyle |G(x,x')| \leq C \, , & \text{if } \, d=1 \, , \\[3pt]
             \displaystyle |G(x,x')| \leq C (-\log |x-x'| \vee 1) \, , & \text{if } \, d = 2 \, ,  \\[3pt]
             \displaystyle |G(x,x')| \leq C |x-x'|^{2-d} \, , & \text{if } \, d \geq 3 \, ,
        \end{array}
    \end{equation*}
    see \textit{e.g.} \citeauthor{green_function_Kim_skellaris} \cite{green_function_Kim_skellaris}. Let $\eta > 0$ be small enough so that $B(x, \eta) \subset D$ and $B(x, \eta) \cap D^{\eta} = \emptyset$. Then, denoting $\underline{D}_{\eta} \coloneqq D \backslash (B(x, \eta) \cup D^{\eta})$, one has:
    \begin{align}
        \mathbb{E} \!  \int_0^{\tau^{0, x}}\! \! \! \mathbf{d}(X_s^{0, x})^{2(\gamma-1)} ds = \int_{B(x, \eta)} \!\!\! \! \mathbf{d}(x')^{2(\gamma-1)} G(&x,x') dx' + \int_{D^{\eta}} \! \mathbf{d}(x')^{2(\gamma-1)} G(x,x') dx' \nonumber \\
        & + \int_{\underline{D}_{\eta}} \! \mathbf{d}(x')^{2(\gamma-1)} G(x,x') dx' \nonumber \, .
    \end{align}
    The function $G(x, \cdot)$ has a singularity in $x$ if $d \ge 2$, but this singularity is integrable so the first integral is finite. The second integral is finite because $\gamma > \frac{1}{2}$. Finally $ \mathbf{d}(\cdot)^{\varepsilon -2} G(x,\cdot)$ is bounded on $\underline{D}_{\eta}$, so the third integral is finite. Note that the bounds on these integrals are uniform in $x \in D$, which concludes de proof.
\end{proof}

\vspace{5pt}

\begin{proof}[Proof of Theorem \ref{thm:Markov_BSDE}]
    By Itô's formula, the process $\delta_t = \mathbf{d}(X_t)$ admits the Itô decomposition \eqref{delta_SDE} with $\delta_0 = \mathbf{d}(x)$, $U \coloneqq \Sigma^{\intercal} \nabla \mathbf{d}(X_{\cdot})$ and $\Xi \coloneqq \nabla \mathbf{d}(X_{\cdot}) \cdot b(X_{\cdot}) + \frac{1}{2} [D^2\mathbf{d}: \Sigma \Sigma^{\intercal}] (X_{\cdot})$. Note that $\Sigma$ is bounded and $|\nabla \mathbf{d}|=1$ in $D^{\overline{\eta}}$, so $|U_t| \geq \sigma_0$ a.s. \hspace{-7pt} when $\delta_t \in [-\overline{\eta}, \overline{\eta}]$. By Assumption \eqref{Lpz_condition_b_sigma}, $b$ has linear growth in $x$, so $b(X_{\cdot})$ is a.s. \hspace{-8pt} bounded for $t \leq {\rm{T}}^{\delta}$. Therefore, Assumption \ref{d_assumptions} holds and the existence of a minimal solution follow from Theorems \ref{existence_thm} and \ref{uniqueness_thm}. Let $D' \subset D$ be a bounded subdomain with $C^2$-regular boundary, and $\theta = \{ t \ge 0 : X_t \notin D' \}$. Then $\delta \in \mathbb{S}^{\infty}_{\theta}$ and $\delta^{\gamma-1} \in \mathbb{H}^{{\rm{BMO}}}_{\theta}$ for all $\gamma > \frac{1}{2}$ by Lemma \ref{lemma_popier_dist}. We may show by the same argument that $\ell(\delta)\delta^{\gamma-1} \in \mathbb{H}^{{\rm{BMO}}}_{\theta}$. We conclude by applying Theorems \ref{existence_thm} and \ref{uniqueness_thm}.
\end{proof}

\vspace{5pt}

\begin{proof}[Proof of Proposition \ref{connection_PDE_prop}]
    Let $f_t(y, z) = -\frac{1}{p}|z|^p - \lambda y + g(X_t)$ and note that Assumption \ref{driver_assumptions} holds. Denote by $u^R$ the unique solution in $W^{2, k}(D)$ of the PDE \eqref{HJB_singular} satisfying $u(x) = R$ when $x \in \partial D$. Let $Y^{x, R}_t \coloneqq u^R(X_t)$ and $Z^{R}_t \coloneqq \nabla u^R(X_t)$. Then by \citeauthor{kobylanski} \cite[Theorem~3.18]{kobylanski}, $(Y^{R}, Z^{R})$ is the unique solution of\hspace{.5pt} ${\rm{BSDE}}_{R}({\rm{T}}^{\delta}, f)$. We conclude the second claim by sending $R$ to $\infty$. Finally, if $\gamma > \frac{1}{2}$ and $\kappa \in \mathcal{T}_0^{{\rm{T}}^{\delta}}$, then we have:
    \begin{equation*}
        \| Z \delta^{\gamma -1} \|_{{\rm{BMO}}, {\rm{T}}^{\delta}} \leq \| Z \delta^{r + 1} \|_{\infty}  \, \| \delta^{\gamma -1} \|_{{\rm{BMO}}, {\rm{T}}^{\delta}}  \,,
    \end{equation*}
    which is finite by Lemma \ref{lemma_popier_dist}.
\end{proof}

\vspace{5pt}

\medskip
\begingroup
{ \footnotesize
\printbibliography}
\endgroup

\end{spacing}

\end{document}